\documentclass[11pt,reqno]{amsart}
\calclayout
\usepackage{fullpage,enumerate}
\usepackage{amsmath}
\usepackage{amssymb,amsthm,xypic,stmaryrd,graphicx,mathtools,mathrsfs}
\usepackage[all]{xy}
\usepackage{bbm, dsfont}
\usepackage{boxedminipage,xcolor}
\usepackage[shortlabels]{enumitem}
\usepackage{thmtools}
\usepackage{slashed}
\usepackage{mathtools}
\usepackage{thm-restate}
\usepackage{tikz-cd}
\usepackage{empheq}
\usepackage{enumitem}
\usepackage[colorlinks=true,linkcolor=blue,citecolor=blue]{hyperref}
\usetikzlibrary{arrows.meta}
\tikzcdset{arrow style=tikz}
\numberwithin{equation}{section}
\DeclareMathOperator{\supp}{supp}

\DeclareMathOperator{\Ind}{Ind}

\DeclareMathOperator{\op}{op}

\DeclareMathOperator{\ind}{index}

\DeclareMathOperator{\Hom}{Hom}
\DeclareMathOperator{\End}{End}

\DeclareMathOperator{\U}{U}
\DeclareMathOperator{\Spin}{Spin}

\DeclareMathOperator{\Tr}{Tr}
\DeclareMathOperator{\tr}{tr}
\DeclareMathOperator{\str}{str}
\DeclareMathOperator{\Str}{Str}

\newcommand{\beq}[1]{\begin{equation} \label{#1}}
\newcommand{\eeq}{\end{equation}}

\newcommand{\bea}{\begin{eqnarray}}
\newcommand{\eea}{\end{eqnarray}}

\begin{document}

\theoremstyle{plain}
\newtheorem{theorem}{Theorem}[section]
\newtheorem{thm}{Theorem}[section]
\newtheorem{assumption}[theorem]{Assumption}
\newtheorem{lemma}[theorem]{Lemma}
\newtheorem{proposition}[theorem]{Proposition}
\newtheorem{prop}[theorem]{Proposition}
\newtheorem{corollary}[theorem]{Corollary}
\newtheorem{conjecture}[theorem]{Conjecture}
\newtheorem{question}[theorem]{Question}

\theoremstyle{definition}
\newtheorem{convention}[theorem]{Convention}
\newtheorem{definition}[theorem]{Definition}
\newtheorem{defn}[theorem]{Definition}
\newtheorem{example}[theorem]{Example}
\newtheorem{remark}[theorem]{Remark}
\newtheorem*{remark*}{Remark}
\newtheorem*{overview*}{Overview}
\newtheorem*{results*}{Results}
\newtheorem{rem}[theorem]{Remark}

\newcommand{\C}{\mathbb{C}}
\newcommand{\R}{\mathbb{R}}
\newcommand{\Z}{\mathbb{Z}}
\newcommand{\N}{\mathbb{N}}
\newcommand{\Q}{\mathbb{Q}}

\newcommand{\Supp}{{\rm Supp}}

\newcommand{\field}[1]{\mathbb{#1}}
\newcommand{\bZ}{\field{Z}}
\newcommand{\bR}{\field{R}}
\newcommand{\bC}{\field{C}}
\newcommand{\bN}{\field{N}}
\newcommand{\bT}{\field{T}}
\newcommand{\cB}{{\mathcal{B} }}
\newcommand{\cK}{{\mathcal{K} }}
\newcommand{\cF}{{\mathcal{F} }}
\newcommand{\cDD}{{\mathcal{D} }}
\newcommand{\cO}{{\mathcal{O} }}
\newcommand{\cE}{\mathcal{E}}
\newcommand{\cS}{\mathcal{S}}
\newcommand{\cN}{\mathcal{N}}
\newcommand{\cl}{\mathcal{L}}
\newcommand{\sL}{\mathscr{L}}
\newcommand{\sS}{\mathscr{S}}
\newcommand{\sA}{\mathscr{A}}
\newcommand{\sB}{\mathscr{B}}
\newcommand{\sR}{\mathscr{R}}

\newcommand{\KK}{K \! K}
\newcommand{\norm}[1]{\| #1\|}
\newcommand{\Spinc}{\Spin^c}
\newcommand{\spartial}{\slashed{\partial}}
\newcommand{\HH}{{\mathcal{H} }}
\newcommand{\Hpi}{\HH_{\pi}}

\newcommand{\DNR}{D_{N \times \R}}


\def\kt{\mathfrak{t}}
\def\kk{\mathfrak{k}}
\def\kp{\mathfrak{p}}
\def\kg{\mathfrak{g}}
\def\kh{\mathfrak{h}}
\def\so{\mathfrak{so}}
\def\cut{c}

\newcommand{\ddt}{\left. \frac{d}{dt}\right|_{t=0}}
\newcommand{\todoa}[1]{\color{red}\textbf{TO DO$^A$: }{#1 }\color{black}}
\newcommand{\todob}[1]{\color{teal}\textbf{TO DO$^B$: }{#1 }\color{black}}
\newcommand{\todoc}[1]{\color{olive}\textbf{TO DO$^C$: }{#1 }\color{black}}
\newcommand{\todod}[1]{\color{blue}\textbf{TO DO$^C$: }{#1 }\color{black}}
\newcommand{\todoe}[1]{\color{green}\textbf{TO DO$^C$: }{#1 }\color{black}}

\newcommand{\esta}{\color{red}\textbf{estimate $1$}\color{black}}
\newcommand{\estb}{\color{teal}\textbf{estimate $2$}\color{black}}
\newcommand{\estc}{\color{olive}\textbf{estimate $3$}\color{black}}
\newcommand{\estd}{\color{blue}\textbf{estimate $4$}\color{black}}
\newcommand{\este}{\color{green}\textbf{estimate $5$}\color{black}}

\title[Higher localised $\widehat{A}$-genera for proper actions and applications]{Higher localised $\widehat{A}$-genera for proper actions and applications}

\author{Hao Guo}
\address[Hao Guo]{Yau Mathematical Sciences Center, Tsinghua University}
\email{haoguo@mail.tsinghua.edu.cn}
\author{Varghese Mathai}
\address[Varghese Mathai]{School of Mathematical Sciences, University of Adelaide}
\email{mathai.varghese@adelaide.edu.au}

\thanks{H.G. and V.M. were partially supported by funding from the Australian Research Council, through the Discovery Project Grant DP200100729. V.M. was supported by the Australian Research Council, through the Australian Laureate Fellowship FL170100020. H.G. was partially supported by NSF DMS-2000082. The authors thank Hang Wang, Zhizhang Xie, and Guoliang Yu for their helpful comments. H.G. is grateful for useful feedback from the audience at the 2021 NUS Conference on Index Theory and Related Topics, August 21-25, 2021}


\maketitle

\begin{abstract}
For a finitely generated discrete group $\Gamma$ acting properly on a spin manifold $M$, we formulate new topological obstructions to $\Gamma$-invariant metrics of positive scalar curvature on $M$ that take into account the cohomology of the classifying space $\underline{B}\Gamma$ for proper actions. 

In the cocompact case, this leads to a natural generalisation of Gromov-Lawson's notion of higher $\widehat{A}$-genera to the setting of proper actions by groups with torsion. It is conjectured that these invariants obstruct the existence of $\Gamma$-invariant positive scalar curvature on $M$. For classes arising from the subring of $H^*(\underline{B}\Gamma,\R)$ generated by elements of degree at most $2$, we are able to prove this, under suitable assumptions, using index-theoretic methods for projectively invariant Dirac operators and a twisted $L^2$-Lefschetz fixed-point theorem involving a weighted trace on conjugacy classes. The latter generalises a result of Wang-Wang \cite{Wang-Wang} to the projective setting. In the special case of free actions and the trivial conjugacy class, this reduces to a theorem of Mathai \cite{MathaiTwisted}, which provided a partial answer to a conjecture of Gromov-Lawson on higher $\widehat{A}$-genera.

If $M$ is non-cocompact, we obtain obstructions to $M$ being a partitioning hypersurface inside a non-cocompact $\Gamma$-manifold with non-negative scalar curvature that is positive in a neighbourhood of the hypersurface. Finally, we define a quantitative version of the twisted higher index, as first introduced in \cite{GXY3}, and use it to prove a parameterised vanishing theorem in terms of the lower bound of the total curvature term in the square of the twisted Dirac operator.

\end{abstract}
\vspace{1cm}
\tableofcontents

\section{Introduction}
\label{sec intro}
In this paper we develop the theory of higher indices of projectively equivariant Dirac operators with respect to proper actions of discrete groups, and relate this to numerical invariants that are computable in terms of characteristic classes. In the spin setting, these results can be applied to give new obstructions to metrics of positive scalar curvature that are invariant under the proper action of a discrete group. This generalises the obstructions provided by the higher $\widehat{A}$-genera of Gromov and Lawson \cite{Gromov-Lawson2} to the setting of proper actions.

On the operator-algebraic side, where the higher index resides, we consider appropriately weighted traces on twisted group $C^*$-algebras. These traces generalise the traces associated to conjugacy classes in the setting of ordinary group $C^*$-algebras by taking into account a $\mathrm{U}(1)$-valued group $2$-cocycle. On the geometric side, there is a corresponding algebra of projectively invariant operators of an appropriate trace class, in which the heat operator associated to a projectively invariant Dirac operator lies. The traces on these two sides are related through a weighted $L^2$-Lefschetz fixed-point theorem that generalises a result of Wang-Wang \cite[Theorem 6.1]{Wang-Wang}.

The special case where the group $\Gamma$ is torsion free and was considered in \cite{CGM, MathaiTwisted}, and led to a proof that the higher $\widehat{A}$-genera of $M$ arising from the subring of $H^*(B\Gamma,\R)$ generated by elements of degree at most $2$ are obstructions to positive scalar curvature on $M/\Gamma$. This gave a partial answer to a conjecture of Gromov-Lawson \cite[Conjecture]{Gromov-Lawson2}; see also \cite[Conjecture 2.1]{Rosenberg3}.

Now let us turn to the general case where an arbitrary discrete group $\Gamma$, possibly with torsion, acts properly on a manifold $M$. Denote by $\underline{B}\Gamma$ the classifying space for proper actions \cite{BCH}, and let
$$f\colon M/\Gamma\to\underline{B}\Gamma$$
be the classifying map for the action of $\Gamma$ on $M$. For any integer $m\geq 0$, let $\alpha\in Z^m(\underline{B}\Gamma,\mathbb{R})$. Let 
\begin{equation}
\label{eq omega}
\omega\in\Omega^m(M)	
\end{equation}
be the $\Gamma$-invariant lift of a differential form on the orbifold $M/\Gamma$ belonging to the class $f^*[\alpha]$ in the orbifold de Rham cohomology group $H^m_{\textnormal{dR}}(M/\Gamma)$. (For background on orbifold de Rham cohomology, see for example \cite[chapter 2]{AdemLeidaRuan}.)

If the quotient $M/\Gamma$ is compact, then for each $g\in\Gamma$, the centraliser $Z^g$ of $g$ acts properly and cocompactly on the fixed-point submanifold $M^g$. Let $c^g$ be a cut-off function for this action, i.e. $c$ is non-negative and satisfies
$$\sum_{k\in Z^g}c^g(k^{-1}x)=1$$
for any $x\in M^g$. We define the \emph{higher localised $\widehat{A}$-genus of $M$} with respect to $\omega$ to be
\begin{equation}
\label{eq higher localised A hat}
\widehat A_{g}(M,\omega)\coloneqq\int_{M^g}c^g\cdot\frac{\widehat A(M^g)\cdot\omega|_{M^g}}{\det(1-g e^{-R^{\mathcal{N}}/2\pi i})^{1/2}},
\end{equation}
where $R^{\cN}$ is the curvature of the Levi--Civita connection restricted to the normal bundle $\cN$ of $M^g$ in $M$ with respect to a $\Gamma$-invariant Riemannian metric on $M$.

We are led to the following conjecture:\\


\begin{conjecture}
\label{conj}
	If a proper, cocompact, connected $\Gamma$-spin manifold $M$ admits a $\Gamma$-invariant Riemannian metric of positive scalar curvature, then all of the higher localised $\widehat{A}$-genera of $M$ vanish. That is, if $\omega$ is any closed $m$-form constructed as above, then for all $g\in\Gamma$, we have
	$$\widehat A_{g}(M,\omega)=0.$$
\end{conjecture}
\vspace{0.3cm}
\begin{remark}
We give evidence for the validity of this conjecture in Theorem \ref{thm main1} and Corollary \ref{cor} by proving it for all $\omega$ arising from the cohomology ring generated by $f^*H^1(\underline{B}\Gamma,\R)\cup f^*H^2(\underline{B}\Gamma,\R)$, under the assumption that $M^g$ is connected and $g$ is a regular element for the relevant group multiplier on $\Gamma$ (see subsection \ref{subsec multipliers}). In view of Remark \ref{rem Moscovici-Wu}, it seems likely that, in the special case of such $\omega$ and regular $g$, Conjecture \ref{conj} is implied by the Baum-Connes conjecture, without any growth conditions on $(g)$.
\end{remark}
\begin{remark}
Conjecture \ref{conj} is a special case of a possible generalisation of the Gromov-Lawson-Rosenberg conjecture in $KO$-theory \cite{Rosenberg3}. We will not discuss this here, since our results are not in this direction.
\end{remark}

The following is an immediate corollary to Conjecture \ref{conj}:
\begin{corollary}
\label{cor conj}
Suppose $\underline{E}\Gamma$ is a proper, cocompact $\Gamma$-spin manifold. Then there is no $\Gamma$-invariant metric of positive scalar curvature on $\underline{E}\Gamma$.
\end{corollary}

Indeed, suppose such a Riemannian metric existed. Letting $\omega$ be the associated volume form and $g=e$, the quantity \eqref{eq higher localised A hat} is $\widehat{A}_e(\underline{E}\Gamma,\omega)=\int_M c\omega$, where $c$ is a cut-off function for the $\Gamma$-action on $M$. This is positive, which contradicts Conjecture \ref{conj}.


We give an index-theoretic approach to a special case of Conjecture \ref{conj} as follows. From the  data above, we construct a twisted Dirac operator that is invariant under a projective representation of $\Gamma$. We show that the integral \eqref{eq higher localised A hat} arises naturally as certain weighted traces of the associated heat operator, and relate this to the higher index of the twisted Dirac operator via a weighted trace map on the twisted group algebra. We prove:

\begin{theorem}
\label{thm main1}
	Let $\Gamma$ be a finitely generated group, and let $M$ be a connected, $\Gamma$-equivariantly spin manifold such that $M/\Gamma$ is compact and $H^1(M)=0$. Let $\omega\in\Omega^2(M)$ be as in \eqref{eq omega}, and let $D$ be the associated projectively invariant Dirac operator on $M$ from Definition \ref{def projective Dirac}. Let $\alpha$ and $\sigma$ be the multipliers constructed in subsection \ref{subsec multipliers}, and suppose that $g\in\Gamma$ is an $\alpha$-regular element, in the sense of \eqref{eq alpha regular}, whose conjugacy class $(g)$ has polynomial growth.
%
%
	\begin{enumerate}[leftmargin=0.29in, label=(\roman*)]
	\item Suppose $M$ is even-dimensional. Then
	$$\big(\tau^{(g)}_\sigma\big)_*\Ind_{\Gamma,\sigma}(D)=\sum_{j=1}^m\int_{M^g_j}e^{-i\phi_g(x_j)}c^g\cdot\frac{\widehat A(M^g_j)\cdot e^{-\omega/2\pi i}|_{M^g_j}}{\det(1-g e^{-R^{\mathcal{N}_j}/2\pi i})^{1/2}}$$
	where $\Ind_{\Gamma,\sigma}(D)\in K_0(C^*_r(\Gamma,\sigma))$ be the $(\Gamma,\sigma)$-invariant higher index of $D$ from Definition \ref{def higher index}, $\big(\tau^{(g)}_\sigma\big)_*$ is the homomorphism \eqref{eq tau induced}, $M^g_1,\ldots,M^g_m$ are the connected components of $M^g$ that intersect the support of a cut-off function $c^g$ for the action of $Z^g$ on $M^g$, $\phi_g$ is as in \eqref{eq derivative of phi}, $x_j$ is an arbitrary point in $M^g_j$, and $R^{\cN_j}$ is the curvature of the Levi-Civita connection restricted to the normal bundle $\cN_j$ of $M^g_j$ in $M$ and with respect to a $\Gamma$-invariant Riemannian metric on $M$.
	\item If $M$ admits a $\Gamma$-invariant Riemannian metric of positive scalar curvature, then for each integer $k\geq 0$, we have
	\begin{equation}
\sum_{j=1}^m\int_{M^g_j}c^g\cdot\frac{\widehat A(M^g)\cdot\left(\phi_g(x_j)-\frac{\omega}{2\pi}\right)^k|_{M^g_j}}{\det(1-g e^{-R^{\mathcal{N}_j}/2\pi i})^{1/2}}=0,
\end{equation}
where the notation is as explained in part (i).
If $M^g$ is connected, then
$$\widehat{A}_g(M,\omega^k)=0$$
for each $k\geq 0$.\\

	\end{enumerate}
\end{theorem}

\begin{corollary}
\label{cor}
	Let $\Gamma$, $M$, and $g$ be as in the statement of Theorem \ref{thm main1}. If $M^g$ is connected, then Conjecture \ref{conj} holds for $\omega$ arising from the subring of $H^*(\underline{B}\Gamma,\R)$ generated by $f^*H^1(\underline{B}\Gamma,\R)\cup f^*H^2(\underline{B}\Gamma,\R)$.\\
\end{corollary}


Next, we prove two obstruction results when $M/\Gamma$ is non-compact, again via index theory of projectively invariant Dirac operators. First, we show that the higher localised $\widehat{A}$-genera are obstructions to $\Gamma$-invariant Riemannian metrics with positive scalar curvature in a neighbourhood of $M$. More precisely, we prove:

\begin{theorem}
	\label{thm main2}
	Let $M$ be a connected spin manifold on which a discrete group $\Gamma$ acts properly, preserving the spin structure, and suppose $H^1(M)=0$. Let $H$ be a connected $\Gamma$-cocompact hypersurface in $M$ with trivial normal bundle. Let the multipliers $\alpha$ and $\sigma$ be as in subsection \ref{subsection Callias localisation}, and suppose that $g\in\Gamma$ is an $\alpha$-regular element, in the sense of \eqref{eq alpha regular}, whose conjugacy class $(g)$ has polynomial growth. 
	
	Suppose $M$ admits a complete $\Gamma$-invariant Riemannian metric whose scalar curvature is non-negative everywhere on $M$ and positive in a neighbourhood of $H$. Let $\omega$ be as above in \eqref{eq higher localised A hat}. Then if $H^g$ is connected,
$$\widehat A_{g}(H,\omega^k|_H)=0$$
 	for each integer $k\geq 0$. In particular,
 		$$\sum_{j=1}^m\int_{H^g_j}e^{-i\phi_g(x_j)}c^g\cdot\frac{\widehat A(H^g_j)\cdot e^{-\omega/2\pi i}|_{H^g_j}}{\det(1-g e^{-R^{\mathcal{N}_j}/2\pi i})^{1/2}}=0,$$
 		where $g\in\Gamma$ is a $\sigma$-regular element, $H^g_1,\ldots,H^g_m$ are the connected components of the fixed-point set $H^g$ that intersect the support of a cut-off function $c^g_H$ for the $Z^g$-action on $H^g$, and $\cN_j$ is the normal bundle of $H^g_j$ in $H$.

\end{theorem}
The proof of this theorem uses a Callias-type index theorem for projectively invariant Dirac operators, which is discussed in section \ref{sec Callias obstr}.

\begin{remark}
Instead of the assumption that $(g)$ has polynomial growth, the conclusions of Theorem \ref{thm main1} and Theorem \ref{thm main2} also hold whenever the trace $\tau^{(g)}_{\sigma^s}\colon\mathbb{C}^{\sigma^s}\Gamma\to\mathbb{C}$ extends continuously to $C^*_r(\Gamma,\sigma^s)$ for all $s$ in a small interval $(0,\delta)$.
\end{remark}

Second, we show that the projectively invariant higher index is compatible with the framework of quantitative $K$-theory \cite{Oyono2}, which is a refinement of operator $K$-theory in the context of geometric $C^*$-algebras. As shown in subsection \ref{subsec qkt}, this allows us to formulate a quantitative notion of projectively invariant higher index, generalising that defined in \cite{GXY3}. With respect to this index, we obtain a parameterised vanishing theorem where the vanishing propagation depends on the lower bound of the twisting curvature of the projective Dirac operator as well as the scalar curvature. More precisely, we prove:
\begin{theorem}
\label{thm main3}
	Fix $0<\varepsilon<\frac{1}{20}$ and $N\geq 7$. There exists a constant $\lambda_0$ such that the following holds. Let $M$ be a smooth spin Riemannian manifold equipped with a proper isometric action by a discrete group $\Gamma$, and suppose that $H^1(M)=0$. Let $p\colon M\to M/\Gamma$ be the projection, $\kappa$ the scalar curvature on $M$, and $\cS\to M$ the spinor bundle.
	
	Let $\omega\in\Omega^2(M)$ be as in \eqref{eq omega}, and let $D^s$ be the associated projectively invariant Dirac operator on $M$ from Definition \ref{def projective Dirac}, acting on sections of the bundle $\cS_\sL$. Let $\kappa$ denote the scalar curvature on $M$, and let $c$ denote Clifford multiplication. If the estimate
	$$\kappa+4isc(\omega)\geq C_s$$ 
	holds for some positive constant $C_s$, then the quantitative $(\Gamma,\sigma^s)$-equivariant higher index of $D^s$ on $L^2(\cS_\sL)$ at scale $r$ vanishes for all $r\geq\frac{\lambda_0}{\sqrt{C_s}}$:
	$$\Ind_{\Gamma,\sigma^s, L^2}^{\varepsilon,r,N}(D^s)=0\in K_*^{\varepsilon,r,N}(C^*(M;L^2(\cS_\sL))^{\Gamma,\sigma^s}),$$
	where the algebra $C^*(M;L^2(\cS_\sL))^{\Gamma,\sigma^s}$ is as in Definition \ref{def twisted Roe}, and $\Ind_{\Gamma,\sigma^s, L^2}^{\varepsilon,r,N}(D^s)$ is defined as in subsection \ref{subsec quantitative higher ind}.
	\end{theorem}
\begin{remark}
When $s=0$, Theorem \ref{thm main3} reduces to \cite[Theorem 1.1]{GXY3}.
\end{remark}

The paper is organised as follows. In section \ref{sec prelim}, we formulate the preliminary definitions and properties of the relevant operator algebras and projectively invariant operators. In section \ref{sec cocompact obstr}, we prove Theorem \ref{thm main1} and Corollary \ref{cor}. In section \ref{sec Callias obstr}, we develop some Callias-type index theory in the projective setting and use it to prove Theorem \ref{thm main2}. In section \ref{sec quantitative obstr}, we formulate the quantitative twisted higher index and prove Theorem \ref{thm main3}.

\hfill\vskip 0.3in
\section{Preliminaries}
\label{sec prelim}
We first fix some notation and recall the necessary operator-algebraic and geometric terminology we will need.
\vspace{0.1in}
\subsection{Notation}
\label{subsec notation}
\hfill\vskip 0.05in
	\noindent
	For $X$ a Riemannian manifold, we write $B(X)$, $C_b(X)$, $C_0(X)$, and $C_c(X)$ to denote the $C^*$-algebras of complex-valued functions on $X$ that are, respectively: bounded Borel, bounded continuous, continuous and vanishing at infinity, and continuous with compact support. If $S\subseteq X$ is a Borel subset, we write $\mathbbm{1}_S$ for the associated characteristic function.
	
	
	For any $C^*$-algebra $A$, denote its unitization by $A^+$. If $\cE$ is a Hilbert module over $A$, let $\cB(\cE)$ and $\cK(\cE)$ denote the $C^*$-algebras of bounded adjointable and compact operators on $\cE$ respectively.
	
	For an element $g$ of a group $G$, we use $Z^g$ to denote the centraliser of $g$ in $G$.
	
	
	
	
	
\vspace{0.1in}

\subsection{Multipliers and projective representations}
\label{subsec multipliers}
\hfill\vskip 0.05in
\noindent Let $\Gamma$ be a discrete group.
\begin{definition}
\label{def multiplier}
A \emph{multiplier} on $\Gamma$ is a map $\sigma\colon\Gamma\times\Gamma\to \mathrm{U}(1)$ satisfying
\begin{enumerate}[(i)]
\item $\sigma(\gamma,\mu)\sigma(\gamma\mu,\delta)=\sigma(\gamma,\mu\delta)\sigma(\mu,\delta)$;
\item $\sigma(\gamma,\gamma^{-1})=\sigma(\gamma^{-1},\gamma)=1$,
\end{enumerate}
for all $\gamma,\mu,\delta\in\Gamma$, where $e$ is the identity element in $\Gamma$.
\end{definition}

In other words, a multiplier on $\Gamma$ is an element of $Z^2(\Gamma,\U(1))$, i.e. a $U(1)$-valued group $2$-cocycle, satisfying the additional normalisation condition (ii). This condition is slightly stronger than the normalisation requirement in \cite{MathaiTwisted}, namely $\sigma(e,\gamma)=\sigma(\gamma,e)=1$, and we have adopted it to simplify some calculations. Nevertheless, every multiplier in the sense of \cite{MathaiTwisted} is cohomologous to a multiplier in our sense. Observe that given a multiplier $\sigma$, its pointwise complex conjugate $\bar\sigma$ is also a multiplier.

We will be concerned specifically with multipliers that arise from proper $\Gamma$-actions on manifolds in the following way.
%
Let $M$ be a smooth, connected manifold equipped with a proper action by a discrete group $\Gamma$, preserving the spin structure. Suppose $M/\Gamma$ is compact and that $H^1(M)=0$. Let $\underline{B}\Gamma$ be the classifying space for proper $\Gamma$-actions \cite{BCH}, and $f\colon M/\Gamma\to \underline B\Gamma$ the classifying map for $M$. 

Let $[\beta]\in H^2(\underline B\Gamma,\R)$ be a $2$-cocycle on $\underline B\Gamma$, and $\omega_0$ a closed $2$-form on $M$ representing the de Rham cohomology class $f^*[\beta]\in H^2_{\textnormal{dR}}(M/\Gamma)$ on the orbifold $M/\Gamma$. Since $[\beta]$ lifts trivially to $\underline E\Gamma$, the lift $\omega$ of $\omega_0$ to $M$ is exact, so there exists a one-form $\eta$ (not necessarily $\Gamma$-invariant) such that 
\begin{equation}
\label{eq omega}
\omega=d\eta.
\end{equation}
Since $\omega$ is $\Gamma$-invariant, $d(\gamma^*\eta-\eta)=\gamma^*\omega-\omega=0$ for all $\gamma\in\Gamma$. Thus $\gamma^*\eta-\eta$ is a closed $1$-form on $M$, and hence exact by assumption. It follows that there exists a family
$$\phi\coloneqq\{\phi_\gamma\colon\gamma\in\Gamma\}$$
of smooth functions on $M$ such that
\begin{equation}
\label{eq derivative of phi}	
\gamma^*\eta-\eta=d\phi_\gamma.
\end{equation}
This implies that for any $\gamma,\gamma'\in\Gamma$,
\begin{equation}
\label{cond 1}
d(\phi_\gamma+\gamma^{-1}\phi_{\gamma'}-\phi_{\gamma'\gamma})=0.
\end{equation}
By way of normalisation, we may assume that there exists some $x_0$ such that 
\begin{equation}
\label{cond 2}
\phi_\gamma(\gamma^{-1}x_0)=0
\end{equation}
for each $\gamma\in\Gamma$. This, together with \eqref{cond 1}, implies that $\phi_e\equiv 0$ and that the formula
\begin{equation}
\label{eq alpha}
\alpha_\phi(\gamma,\gamma')=\frac{1}{2\pi}(\phi_{\gamma'}(x_0)+\phi_\gamma(\gamma'x_0)-\phi_{\gamma\gamma'}(x_0))
\end{equation}
defines an element of $Z^2(\Gamma,\R)$. For each $s\in\R$, we have an associated $\U(1)$-valued $2$-cocycle
\begin{equation}
\label{eq sigma}
\sigma^s_{\phi}(\gamma,\gamma')\coloneqq e^{2\pi is\alpha_\phi(\gamma,\gamma')},
\end{equation}
which is a multiplier the sense of Definition \ref{def multiplier}. When it is clear from context, we will use the following abbreviations: 
$$\alpha=\alpha_\phi,\qquad\sigma^s=\sigma_\phi^s,\qquad\sigma=\sigma^1.$$

It can be shown that for a given class $[\beta]$, different choices of $\omega$, $\eta$, and $\phi$ all lead to cohomologous $\sigma$. Nevertheless, it is useful to make specific choices, as the numerical obstructions we compute in Theorems \ref{thm main1} and \ref{thm main2} are expressed in terms of $\phi$.
\begin{remark}
\label{rem fixed points}
Restricting \eqref{eq derivative of phi} to the fixed-point submanifold $M^\gamma$, one sees that
$$\gamma^*\eta|_{M^\gamma}-\eta|_{M^\gamma}=d_{M^\gamma}\phi_\gamma=0,$$
hence on any connected component $M^\gamma$, the function $\phi_\gamma$ is constant. 
\end{remark}

\begin{definition}
\label{def proj action}
Let $E\to M$ be a $\Gamma$-equivariant $\C$-vector bundle. For each $\gamma\in\Gamma$ and $s\in\R$, define the unitary operators $U_\gamma$, $S^s_\gamma$, and $T^s_\gamma$ on $L^2(E)$ by:
\begin{itemize}
\item $U_\gamma u(x)=\gamma u(\gamma^{-1}x)$;
\item $S_\gamma^s u=e^{is\phi_\gamma}u$;
\item $T_\gamma^s=U_\gamma\circ S_\gamma^s$,
\end{itemize}
for $u\in L^2(E)$ and $x\in M$. We refer to $T^s$ as a \emph{projective action} on $L^2(E)$.
\end{definition}
Note that for any $\gamma,\gamma'\in\Gamma$ and $s\in\R$, we have
	$$T_\gamma^s T_{\gamma'}^s=\sigma^s(\gamma,\gamma')T_{\gamma\gamma'}^s.$$
Thus for each $s$, the map $T^s\colon\Gamma\to\U(L^2(E))$ given by $T^s(\gamma)=T^s_\gamma$ defines a projective representation of $\Gamma$ in the sense of subsection \ref{subsec twisted algebras} below. An operator on $L^2(E)$ that commutes with $T^s$ is said to be \emph{$(\Gamma,\sigma^s)$-invariant} or simply \emph{projectively invariant} if no confusion arises.

Now suppose $M$ is $\Gamma$-equivariantly spin, equipped with a $\Gamma$-invariant Riemannian metric. Let $E=\cS$ be the spinor bundle and $\slashed{\partial}$ the spin-Dirac operator. We can obtain a $(\Gamma,\sigma^s)$-invariant Dirac operator as follows. Let $\sL\to M$ be a $\Gamma$-equivariantly trivial line bundle. For each $s\in R$, consider the Hermitian connection
$$\nabla^{\sL,s}\coloneqq d+is\eta$$
on $\sL$. Equip $\cS_{\sL}=\cS\otimes\sL$ with the obvious $\mathbb{Z}_2$-grading. Then we have:
\begin{definition}
\label{def projective Dirac}
For each $s\in\R$, we will refer to the operator
\begin{equation}
\label{eq twisted Dirac}
D^s\coloneqq\slashed{\partial}\otimes\nabla^{\sL,s}\colon L^2(\cS_{\sL})\to L^2(\cS_{\sL})
\end{equation}
as the \emph{$(\Gamma,\sigma^s)$-invariant Dirac operator}, or simply a \emph{projectively invariant} Dirac operator if no confusion arises.
We write $D=D^1$.
\end{definition}
One can verify, as in \cite[Lemma 3.1]{MathaiChatterji}, that $D^s$ commutes with the projective $(\Gamma,\sigma^s)$-action defined by applying Definition \ref{def proj action} to $E=\cS_\sL$.

\vspace{0.1in}
\subsection{Twisted group $C^*$-algebras}
\label{subsec twisted algebras}
%
Given a multiplier $\sigma$ on a discrete group $\Gamma$, a \emph{unitary $(\Gamma,\sigma)$-representation}, or \emph{projective representation}, of $\Gamma$ is a map 
$$T\colon\Gamma\to\mathrm{U}(H),\qquad \gamma\mapsto T_\gamma,$$
for some Hilbert space $H$, such that
\begin{align*}
T_e&=\textnormal{Id}_H;\\
T_{\gamma_1}\circ T_{\gamma_2}&=\sigma(\gamma_1,\gamma_2)T_{\gamma_1\gamma_2}
\end{align*}
for all $\gamma_1,\gamma_2\in\Gamma$. 

The \emph{twisted group algebra} $\mathbb{C}^\sigma\Gamma$ is the associative $*$-algebra over $\C$ with basis $\{\bar\gamma\colon\gamma\in\Gamma\}$, where the multiplication and $*$-operation are given on basis elements by
$$\bar{\gamma}_1\cdot\bar{\gamma}_2=\sigma(\gamma_1,\gamma_2)\overline{\gamma_1\gamma_2},\qquad\bar{\gamma}^*=\overline{\gamma^{-1}},$$
and extended linearly and conjugate-linearly respectively. It follows from condition (ii) in Definition \ref{def multiplier} that the identity element of $\mathbb{C}^\sigma\Gamma$ is $\bar e$ and that $\bar\gamma^{-1}=\bar\gamma^*$.


The \emph{reduced twisted group $C^*$-algebra} is constructed as follows. Consider $l^2(\Gamma)$ with its usual basis $\{\delta_\gamma\}_{\gamma\in\Gamma}$, and define a projective representation 
$$\lambda\colon\Gamma\to\U(l^2(\Gamma))$$ 
by the following action on basis elements:
$$\lambda_{\gamma_1}\delta_{\gamma_2}\coloneqq\sigma(\gamma_1,\gamma_1^{-1}\gamma_2)\delta_{\gamma_1\gamma_2}.$$
Then $\lambda$ extends naturally to an injective $*$-representation on $\mathcal{B}(l^2(\Gamma))$, called the \emph{left regular representation} of $\mathbb{C}^\sigma\Gamma$. The reduced twisted group $C^*$-algebra $C^*_r(\Gamma,\sigma)$ is the completion of $\mathbb{C}^\sigma\Gamma$ with respect to the induced norm. When convenient, we will simply abbreviate $\lambda_{\gamma}$ to $\bar\gamma$.

When $\sigma\equiv 1$ is the trivial multiplier, $\C^\sigma\Gamma$ and $C^*_r(\Gamma,\sigma)$ are the ordinary group algebra and reduced group $C^*$-algebra respectively.
\begin{remark}
\label{rem not multiplier}
More generally, twisted group algebras can be formed using an arbitrary group $2$-cocycle instead of a multipler. In this case, we would have $\bar{\gamma}^*=\bar\gamma^{-1}=\bar\sigma(\gamma,\gamma^{-1})\overline{\gamma^{-1}}$. We will use this extra flexibility in Proposition \ref{prop consistent trace} below.
\end{remark}


\vspace{0.1in}
\section{Weighted traces and projective indices}
\label{sec weighted traces}
We now define the traces associated to conjugacy classes, in an appropriate sense, on the algebraic part of the twisted group algebra that we will use in this paper.

When $\sigma\equiv 1$ is the trival multiplier, we recover from $\mathbb{C}^\sigma\Gamma$ the group algebra $\mathbb{C}\Gamma$. In this case, for any conjugacy class $(g)\subseteq\Gamma$, the map
\begin{align}
\label{eq tau g}
	\tau^{(g)}\colon\mathbb{C}
	\Gamma&\to\mathbb{C}\nonumber\\
	\sum_{\gamma\in\Gamma}a_\gamma\gamma&\mapsto\sum_{\gamma\in(g)}a_\gamma
\end{align}
is a trace. When $\sigma$ is non-trivial, this formula ceases in general to define a trace on $\mathbb{C}^\sigma\Gamma$. However, for conjugacy classes of certain \emph{$\sigma$-regular} elements, one can define a trace via a weighted sum determined by $\sigma$.
\begin{definition}
\label{def regular}
Let $\sigma$ be a multiplier on $\Gamma$. An element $g\in\Gamma$ is \emph{$\sigma$-regular} if
$$\sigma(g,z)=\sigma(z,g)$$
for all $z\in Z^g$.
\end{definition}
The following equivalent formulation is useful:
\begin{lemma}
\label{lem equiv reg}
An element $g\in\Gamma$ is $\sigma$-regular if and only if $\bar z^{-1}\bar g\bar z=\bar g$ for any $z\in Z^g$.
\end{lemma}
\begin{proof}
Note that
$$\bar z^{-1}\bar g\bar z=\sigma(z^{-1},gz)\sigma(g,z)\overline{z^{-1}gz}=\sigma(z^{-1},zg)\sigma(g,z)\bar g,$$
where we have used that $zg=gz$, while applying Definition \ref{def multiplier} shows that
$$\bar g=\sigma(z^{-1},z)\sigma(z^{-1}z,g)\bar g=\sigma(z^{-1},zg)\sigma(z,g)\bar g.$$
The two expressions on the right are equal if and only if $g$ is $\sigma$-regular.
\end{proof}

The property of being $\sigma$-regular is invariant under conjugation in $\Gamma$, hence we may speak of $\sigma$-regular conjugacy classes. The next lemma implies that in order to consider any traces at all on $\C^\sigma\Gamma$, it is necessary to deal with $\sigma$-regular elements (see also \cite[Lemma 1.2]{Passman}):

\begin{lemma}
If $g$ is not $\sigma$-regular, then for any trace map $t\colon\mathbb{C}^\sigma\Gamma\to\mathbb{C}$, $t(\overline{g})=0$.
\end{lemma}
\begin{proof}
If $g$ is not $\sigma$-regular, then by Lemma \ref{lem equiv reg} there exists $z\in Z^g$ such that $\bar z^{-1}\bar g\bar z=\lambda\bar g$ for $\lambda\neq 1$. Now
$$\bar z(\bar z^{-1}\bar g)-(\bar z^{-1}\bar g)\bar z=\bar g-\bar z^{-1}\bar g\bar z=(1-\lambda)\bar g.$$
Since $t$ is a trace, $t(1-\lambda)\bar g=0$ and hence $t(\bar g)=0$.	
\end{proof}
To define traces for $\sigma$-regular conjugacy classes, we will use the following weighting function. Define a function $\theta\colon\Gamma\to\U(1)$ by
\begin{equation}
\label{eq theta}
\theta(\gamma)=\begin{cases}
\sigma(g^{-1},k)\sigma(k^{-1},g^{-1}k)&\textnormal{if }\gamma=k^{-1}gk\textnormal{ for some $k\in\Gamma$},\\
1&\textnormal{otherwise.}	
\end{cases}
\end{equation}
To prove that $\theta$ is well-defined, as well as for later calculations, we will make use of the following lemma.
\begin{lemma}
\label{lem invert and swap}
For any multiplier $\sigma$ on $\Gamma$ and $h,k\in\Gamma$, we have
$$\sigma(h,k)\sigma(k^{-1},h^{-1})=1.$$
\end{lemma}

\begin{proof}
Note that
\begin{align*}
\bar e&=\bar k\bar h\bar h^{-1}\bar k^{-1}=\sigma(k,h)\overline{kh}\cdot\overline{h^{-1}}\cdot\overline{k^{-1}}.
\end{align*}
Multiplying and simplifying, this equals
$$\sigma(k,h)\sigma(h^{-1},k^{-1})\bar e.$$
The claim follows by equating coefficients of $\bar e$.
\end{proof}
\begin{proposition}
The function $\theta$ in \eqref{eq theta} is well-defined.	
\end{proposition}
\begin{proof}
First note that if $\gamma=k^{-1}a^{-1}gak$ for some $a\in Z^g$, then
\begin{align}
\label{eq well-defined lhs}
\sigma(k^{-1}a^{-1},g)\sigma(k^{-1}a^{-1}g,ak)\overline{k^{-1}gk}
&=\sigma(k^{-1}a^{-1},g)\sigma(k^{-1}a^{-1}g,ak)\overline{k^{-1}a^{-1}gak}\nonumber\\&=\overline{k^{-1}a^{-1}}\cdot\overline{g}\cdot\overline{ak}\nonumber\\
&=\bar{\sigma}(k^{-1},a^{-1})\bar{\sigma}(a,k)\overline{k^{-1}}\bar{a}^{-1}\bar{g}\bar{a}\bar{k}.\nonumber
\end{align}
Since $a\in Z^g$ and $g$ is $\sigma$-regular, this equals
\begin{align*}
\bar{\sigma}(k^{-1},a^{-1})\bar{\sigma}(a,k)\overline{k^{-1}}\bar g\bar k=\bar{\sigma}(k^{-1},a^{-1})\bar{\sigma}(a,k)\sigma(k^{-1},g)\sigma(k^{-1}g,k)\overline{k^{-1}gk}.
\end{align*}
By Lemma \ref{lem invert and swap}, this equals
\begin{equation*}
\label{eq well-defined rhs}
	\sigma(k^{-1},g)\sigma(k^{-1}g,k)\overline{k^{-1}gk}.
\end{equation*}
Equating coefficients of $\overline{k^{-1}gk}$ shows that $\theta$ is well-defined.
\end{proof}
\begin{remark}
\label{rem theta g}
Observe that $\theta(e)=\theta(g)=1$.
\end{remark}
\begin{remark}
When $\Gamma$ is finite, it is known that the set of $\sigma$-regular conjugacy classes in $\Gamma$ is in bijection with the set of distinct irreducible $(\Gamma,\sigma)$-representations.
\end{remark}
\begin{definition}
\label{def weighted trace}
Suppose that $(g)$ is $\sigma$-regular for some multiplier $\sigma$ on $\Gamma$. Define the \emph{$\sigma$-weighted $(g)$-trace} $\tau^{(g)}_\sigma\colon\mathbb{C}^\sigma\Gamma\to\mathbb{C}$ by
\begin{equation*}
	\sum_{\gamma\in\Gamma}a_\gamma\bar\gamma\mapsto\sum_{\gamma\in(g)}\theta(\gamma)a_\gamma.
\end{equation*}
\end{definition}
We will show that $\tau^{(g)}_\sigma$ is a trace in two steps. Define
$$\sigma'=\sigma d\bar\theta,$$
where $d\bar\theta(\gamma_1,\gamma_2)=\bar\theta(\gamma_1\gamma_2)\theta(\gamma_1)\theta(\gamma_2)$ is a $2$-coboundary. Then $\sigma'$ is a $2$-cocycle cohomologous to $\sigma$. Let $\mathbb{C}^{\sigma'}\Gamma$ be the twisted group algebra defined using $\sigma'$ (see Remark \ref{rem not multiplier}), with basis $\left\{\underline\gamma\colon\gamma\in\Gamma\right\}$. We first prove:
\begin{proposition}
\label{prop consistent trace}
The map $\tau^{(g)}\colon\mathbb{C}^{\sigma'}
	\Gamma\to\mathbb{C}$ defined by
\begin{align*}
	\sum_{\gamma\in\Gamma}a_\gamma\underline\gamma\mapsto\sum_{\gamma\in(g)}a_\gamma
\end{align*}
is a trace.	
\end{proposition}
\begin{proof}
	By definition of the multiplication in $\mathbb{C}^{\sigma'}\Gamma$, we have for any $k\in\Gamma$ that
	\begin{align*}
		\underline{k}^{-1}\underline{g}\underline{k}&=\bar\sigma'(k,k^{-1})\underline{k^{-1}}\underline g\underline k\\
		&=\bar\sigma'(k,k^{-1})\sigma'(g,k)\sigma'(k^{-1},gk)\underline{k^{-1}gk}.
	\end{align*}
	Using the definition of $\sigma'$, this is equal to
$$\bar\sigma(k,k^{-1})\sigma(g,k)\sigma(k^{-1},gk)\theta(e)\bar\theta(k)\bar\theta(k^{-1})\bar\theta(gk)\theta(g)\theta(k)\bar\theta(k^{-1}gk)\theta(k^{-1})\theta(gk)\underline{k^{-1}gk}.$$
Upon cancelling and applying Remark \ref{rem theta g}, together with the definition of $\theta$, this simplifies to $\underline{k^{-1}gk}$. It follows that
\begin{equation}
\label{eq conjugation}	
\underline{k}^{-1}\underline{g}\underline{k}=\underline{k^{-1}gk}
\end{equation}
for any $k\in\Gamma$. By extension, this holds if $g$ is replaced by any $h\in(g)$. Indeed, if $h=m^{-1}gm$, then
	\begin{align*}
	\underline k^{-1}\underline h\,\underline k&=\underline k^{-1}\underline{m^{-1}gm}\underline k\\
	&=\underline k^{-1}\underline m^{-1}\underline g\underline m\,\underline k\\
	&=\bar\sigma(k,k^{-1})\bar\sigma(m,m^{-1})\underline{k^{-1}}\cdot\underline{m^{-1}}\underline{g}\underline{m}\cdot\underline{k}\\
	&=\bar\sigma(k,k^{-1})\bar\sigma(m,m^{-1})\sigma(m,k)\sigma(k^{-1},m^{-1})\underline{k^{-1}}\cdot\underline{m^{-1}}\underline{g}\underline{m}\cdot\underline{k}\\
	&=\bar\sigma(k,k^{-1})\bar\sigma(m,m^{-1})\sigma(m,k)\sigma(k^{-1},m^{-1})\sigma(k^{-1}m^{-1},mk)(\underline{mk})^{-1}\underline{g}\underline{mk}.
	\end{align*}
Applying Lemma \ref{lem invert and swap} and \eqref{eq conjugation}, this is equal to
\begin{align*}
\underline{(mk)^{-1}gmk}&=\underline{k^{-1}hk}.	
\end{align*}

	To see that $\tau^{(g)}$ is a trace, it suffices to show that for any $k,m\in\Gamma$, we have $\text{tr}^{(g)}\underline{k}\underline{m} = \text{tr}^{(g)}\underline{m}\underline{k}$. We may assume that $km=h\in(g)$, so that $\underline{k}\underline{m}=\lambda\underline{h}$ for some $\lambda\in\U(1)$. Then clearly $\text{tr}^{(g)}\underline{k}\underline{m}=\lambda$. On the other hand,
$$\underline m\underline k=\underline k^{-1}(\underline k\underline m)\underline k=\underline k^{-1}(\lambda\underline h)\underline k=\lambda\underline k^{-1}\underline h\,\underline k.$$
By the preceding discussion, this is equal to
$$\lambda\underline{k^{-1}hk},$$
hence $\text{tr}^{(g)}\underline m\underline k=\lambda$.
\end{proof}
We now use Proposition \ref{prop consistent trace} to deduce:
\begin{proposition}
\label{prop special trace}
	$\tau^{(g)}_\sigma\colon\mathbb{C}^\sigma\Gamma\to\mathbb{C}$ is a trace.
\end{proposition}
\begin{proof}
Using that $\sigma'$ is cohomologous to $\sigma$ via the coboundary $d\bar\theta$, one verifies that the map
\begin{align*}
f_\theta\colon\mathbb{C}^\sigma\Gamma&\to\mathbb{C}^{\sigma'}\Gamma,	\\
\bar\gamma&\mapsto\bar\theta(\gamma)\underline{\gamma}
\end{align*}
is an isomorphism of $*$-algebras. We have a commutative diagram
$$\begin{tikzcd}
\mathbb{C}^{\sigma'}\Gamma\ar[r,"\tau^{(g)}"]&\C\\
\mathbb{C}^{\sigma}\Gamma\ar[u,"f_\theta"]\ar[r,"\tau_\sigma^{(g)}"]\ar{r}&\C\ar[u,"="],
\end{tikzcd}
$$
from which it follows that $\tau_\sigma^{(g)}$ is a trace.
\end{proof}
For real-valued group cocycles, the notion of regularity also applies: for any $\alpha\in Z^2(\Gamma,\R)$, we say that $g\in\Gamma$ is \emph{$\alpha$-regular} if
\begin{equation}
\label{eq alpha regular}
\alpha(g,z)=\alpha(z,g)
\end{equation}
for all $z\in Z^g$. For such $g$, let us define a function $\psi\colon\Gamma\to\R$ by
\begin{equation}
\label{eq psi}
\psi(\gamma)=\begin{cases}
\alpha(g^{-1},k)+\alpha(k^{-1},g^{-1}k)&\textnormal{if }\gamma=k^{-1}gk\textnormal{ for some $k\in\Gamma$},\\
0&\textnormal{otherwise.}	
\end{cases}
\end{equation}
As with the function $\theta$ in \eqref{eq theta}, one checks that $\psi$ is well-defined. 

\begin{corollary}
\label{cor trace any s}
Let $\alpha\in Z^2(\Gamma,\R)$ be defined as in \eqref{eq alpha}, and let $\{\sigma^s\}_{s\in\R}$ be the associated family multipliers as in \eqref{eq sigma}. If $g\in\Gamma$ is $\alpha$-regular, then for every $s\in\R$,
\begin{enumerate}[(i)]
	\item the conjugacy class $(g)$ is $\sigma^s$-regular;
	\item the formula
$$\sum_{\gamma\in\Gamma}a_\gamma\bar\gamma\mapsto\sum_{\gamma\in(g)}e^{2\pi is\psi}(\gamma) a_\gamma.$$
defines a trace $\tau^{(g)}_{\sigma^s}\colon\mathbb{C}^{\sigma^s}\Gamma\to\mathbb{C}$, where $\psi$ is as in \eqref{eq psi}.
\end{enumerate}
\end{corollary}
\begin{proof}
By \eqref{eq sigma} and the fact that $g$ is $\alpha$-regular,
$$\sigma^s(g,z)=e^{2\pi is\alpha(g,z)}=e^{2\pi is\alpha(z,g)}=\sigma^s(z,g)$$
for any $z\in Z^g$, which proves (i). For (ii), 
note that
the proofs of Propositions \ref{prop consistent trace} and \ref{prop special trace}, with $\theta$ replaced by 
\begin{equation}
\label{eq theta s}
\theta^s\coloneqq e^{2\pi is\psi}	
\end{equation} 
and $\sigma'$ replaced by $(\sigma^{s})'=\sigma^s d\theta^s$, imply that $\tau^{(g)}_{\sigma^s}$ is a trace for each $s$.
\end{proof}





\vspace{0.1in}
\subsection{Traces on operators}
 	\hfill\vskip 0.05in

\noindent We now define weighted traces on projectively invariant operators of an appropriate class. Let $M$, $E$, $\phi$, and $\sigma$ be as in subsection \ref{subsec multipliers}.
\begin{definition}
A continuous function $c\colon M\to[0,\infty)$ is called a \emph{cut-off function} for the $\Gamma$-action on $M$ if for any $x\in M$ we have
$$\sum_{g\in\Gamma}c(g^{-1}x)=1.$$
\end{definition}
\begin{remark}
For any proper action, a cut-off function always exists. If the action is cocompact, the cut-off function can be chosen to be compactly supported.
\end{remark}
We have the following useful lemma:
\begin{lemma}
\label{lem useful integration properties}
	Suppose $M/\Gamma$ is compact, and let $c$ be any cut-off function.
	\begin{enumerate}[(i)]
	\item Let $f$ be a smooth function on $M/\Gamma$ and $\widetilde{f}$ its lift to $M$. Then
	$$\int_M c(x)\widetilde{f}(x)\,dx=\int_{M/\Gamma}f(z)\,dz.$$
	\item Let $h$ be a continuous function on $M\times M$ that is invariant under the diagonal action of $\Gamma$.
	If $c(x)h(x,y),c(y)h(x,y)$ are integrable on $M\times M$, then
	$$\int_{M\times M}c(x)h(x,y)\,dx\,dy=\int_{M\times M}c(y)h(x,y)\,dx\,dy.$$
	\end{enumerate}
\end{lemma}
\begin{proof}
Following \cite[section 3]{Wang-Wang}, let us define the following special cut-off function. First write $M$ as
$$M=\bigcup_{i\in\mathbb{N}}\Gamma\times_{F_i}U_i,$$
for finite subgroups $F_i$ of $\Gamma_i$ and $F_i$-invariant, relatively compact subsets $U_i\subseteq M$. Then $M/\Gamma$ admits the open cover $\{U_i/F_i\}_{i\in\mathbb{N}}$, along with a subordinate partition of unity $\{\phi_i\}_{i\in\mathbb{N}}$. Write $\widetilde{\phi}_i$ for the $\Gamma$-invariant lift of $\phi_i$ to $M$. For each $i$, define the function $\varphi_i\colon G\times_{F_i}U_i\to[0,\infty)$ by
$$\varphi_i([g,u])\coloneqq
\begin{cases}
\widetilde{\phi}_i(u) & \textnormal{for }g\in F_i,\\
0 & \textnormal{otherwise.}	
\end{cases}
$$
Then $\varphi_i$ extends by zero to an $F_i$-invariant function $\psi_i$ on all of $M$. Define a smooth function $\mathfrak{c}\colon M\to[0,\infty)$ by 
\begin{equation}
\label{eq frak c}
\mathfrak{c}(x)\coloneqq\sum_{i\in\mathbb{N}}\frac{1}{|F_i|}\psi_i(x).	
\end{equation}
By \cite[Lemma 3.9]{Wang-Wang}, $\mathfrak{c}$ is a compactly supported, smooth cut-off function on $M$ whenever the $\Gamma$-action is cocompact.

The case of $c=\mathfrak{c}$, where $\frak c$ is defined as in \eqref{eq frak c}, is a special case of \cite[Lemma 3.10]{Wang-Wang}. Now observe that for any $\Gamma$-invariant function $r$ on $M$, the integral $\int_M c(x)r(x)\,dx$ is independent of the choice of cut-off function $c$. Applying this to $r=\widetilde f$ yields the general case of (i), while letting
$$r(x)=\int_M h(x,y)\,dy,\qquad r(y)=\int_M h(x,y)\,dx$$
yields the general case of (ii).
\end{proof}


\begin{definition}
\label{def g trace}
Let $\sigma$ be a multiplier on $\Gamma$, and let $(g)\subseteq\Gamma$ be a $\sigma$-regular conjugacy class. 
\begin{itemize}[leftmargin=0.29in]
\item A bounded $(\Gamma,\sigma)$-invariant operator $S$ on $L^2(\cS_\sL)$ is said to be of \emph{$(\sigma,g)$-trace class} if for all $\phi_1,\phi_2\in C_c(M)$,
\begin{enumerate}[(i)]
\item the operator $\phi_1 T_{h^{-1}}S\phi_2$ is of trace class for any $h\in(g)$;
\item the sum
\begin{equation}
\label{eq g trace class}
\sum_{h\in(g)}\tr(\phi_1 T_{h^{-1}}S\phi_2	)
\end{equation}
converges absolutely.
\end{enumerate}
\item If $S$ is of $(\sigma,g)$-trace class, we define its \emph{$\sigma$-weighted $(g)$-trace} to be
$$\text{tr}^{(g)}_{\sigma}(S)\coloneqq\sum_{h\in(g)}\theta(h)\cdot\text{tr}(c_1T_{h^{-1}}Sc_2),$$
for some $c_1,c_2\in C_c(M)$ such that $c_1c_2$ is a cut-off function on $M$.
\end{itemize}

\end{definition}
\begin{remark}
\label{rem finite prop trace}
	If $S$ has finite propagation, then for any $\phi\in C_c(M)$,
	there exists a compactly supported continuous function $f$ such that $\phi S=\phi Sf$.  The $\sigma$-weighted $(g)$-trace of $S$ can then be equivalently defined to be
	$$\sum_{h\in(g)}\theta(h)\cdot\text{tr}(cT_{h^{-1}}S),$$
for any cut-off function $c$. Compare \cite[Definition 3.13]{Wang-Wang} and \cite[Definition 3.1]{Hang}.
\end{remark}

\begin{lemma}
\label{lem integral formula}
If $S$ is a $(\Gamma,\sigma)$-invariant $(\sigma,g)$-trace 	class operator, then
$$\textnormal{tr}^{(g)}_\sigma(S)=\sum_{h\in(g)}\theta(h)\int_M e^{-i\phi_h(x)}c(x)\textnormal{Tr}(h^{-1}K_S(hx,x))\,dx,$$
where $c$ is a cut-off function as in Definition \ref{def g trace}, and $K_S$ denotes the Schwartz kernel of $S$.
\end{lemma}
\begin{proof}
	Let $c_1,c_2\in C_c(M)$ such that $c_1 c_2=c$. Note that for any $u\in L^2(\cS_{\sL})$ and $h\in\Gamma$, we have
	\begin{align*}
	(U_h^{-1}Sc_2)u(x)&=\int_M h^{-1}K_S(hx,y)c_2(y)u(y)\,dy.		
	\end{align*}
	Since $T_h^{-1}=T_{h^{-1}}=S_h^{-1}U_h^{-1}$, this implies that
	$$(c_1 T_h^{-1}Sc_2)u(x)=\int_M e^{-i\phi_h(x)}c_1(x)h^{-1} K_S(hx,y)c_2(y)u(y)\,dy.$$
	Taking the trace gives
	\begin{align*}
	\textnormal{tr}(c_1T_{h^{-1}}Sc_2)&=\int_M e^{-i\phi_h(x)}c_1(x)\textnormal{Tr}(h^{-1}K_S(hx,x))c_2(x)\,dx\\
	&=\int_M e^{-i\phi_h(x)}c(x)\textnormal{Tr}(h^{-1}K_S(hx,x))\,dx.
	\end{align*}
	Multiplying by $\theta(h)$ and taking a sum finishes the proof.
\end{proof}

Our task now is to show that $\tr^{(g)}_\sigma$ satisfies the tracial property and that it is independent of the choices of $c_1$, $c_2$, and $c$ made in Definition \ref{def g trace}. This will be carried out in Proposition \ref{prop trace properties}, after we record some preparatory observations in the form of Lemmas \ref{lem special 0} and \ref{lem special useful}.

\begin{lemma}
\label{lem special 0}
Suppose $h=x^{-1}gx$. Then for any $\gamma\in\Gamma$,
$$\sigma(\gamma^{-1}h,\gamma)\sigma(\gamma^{-1},h)=\sigma(\gamma^{-1}x^{-1},g)\sigma(\gamma^{-1}x^{-1}g,x\gamma)\bar{\sigma}(x^{-1},g)\bar\sigma(x^{-1}g,x).$$
\end{lemma}
\begin{proof}
Observe that $\bar{\gamma}^{-1}\bar h\bar\gamma$ is equal to
\begin{align*}
\bar\gamma\overline{x^{-1}gx}\bar\gamma=\bar{\gamma}^{-1}\bar{x}^{-1}\bar{g}\bar{x}\bar\gamma\bar\sigma(x^{-1},g)\bar\sigma	(x^{-1}g,x).
\end{align*}
On the other hand, it is also equal to $\sigma(\gamma^{-1},h)\sigma(\gamma^{-1}h,\gamma)\overline{\gamma^{-1}h\gamma}$, which, can be written as
\begin{align*}
\sigma(\gamma^{-1},h)\sigma(\gamma^{-1}h,\gamma)\overline{\gamma^{-1}x^{-1}}\bar g\overline{x\gamma}\bar\sigma(\gamma^{-1}x^{-1},g)\bar\sigma(\gamma^{-1}x^{-1}g,x\gamma).
\end{align*}
By Lemma \ref{lem invert and swap}, we have $\bar\gamma^{-1}\bar x^{-1}\bar g\bar x\bar\gamma=\overline{\gamma^{-1}x^{-1}}\bar g\overline{x\gamma}$. We conclude by equating coefficients.
\end{proof}

\begin{lemma}
\label{lem special useful}
Let $(g)\subseteq\Gamma$ be a $\sigma$-regular conjugacy class for some multiplier $\sigma$. Then for any $h\in(g)$ and $\gamma\in\Gamma$, we have
$$\theta(\gamma^{-1} h\gamma)\cdot\bar\sigma(\gamma^{-1}h\gamma,\gamma^{-1})=\theta(h)\cdot\bar\sigma(\gamma^{-1},h).$$
\end{lemma}
\begin{proof}
We have
\begin{align*}
\bar\sigma(\gamma^{-1}h\gamma,\gamma^{-1})&={\sigma}(\gamma,\gamma^{-1}h^{-1}\gamma)\\
&=\bar\sigma(\gamma^{-1},h^{-1}\gamma)\\
&=\sigma(\gamma^{-1}h,\gamma).
\end{align*}
The first and third equalities follow from Lemma \ref{lem invert and swap}; the second follows from the fact that $\sigma(\gamma,\gamma^{-1}h^{-1}\gamma)\sigma(\gamma^{-1},h^{-1}\gamma)=1$.
Thus it suffices to show that
\begin{equation}
\label{eq long}
\theta(\gamma^{-1}h\gamma)\bar\theta(h)\sigma(\gamma^{-1}h,\gamma)\sigma(\gamma^{-1},h)=1.
\end{equation}
For this, let $h=k^{-1}gk$. Then by Lemma \ref{lem special 0},
$$\sigma(\gamma^{-1}h,\gamma)\sigma(\gamma^{-1},h)=\sigma(\gamma^{-1}k^{-1},g)\sigma(\gamma^{-1}k^{-1}g,k\gamma)\bar{\sigma}(k^{-1},g)\bar\sigma(k^{-1}g,k).$$
It follows from the definition of $\theta$ that the left-hand side of \eqref{eq long} equals
\begin{multline*}
\sigma(g^{-1},k\gamma)\sigma(\gamma^{-1}k^{-1},g^{-1}k\gamma)\bar\sigma(g^{-1},k)\bar\sigma(k^{-1},g^{-1}k)\\\cdot\sigma(\gamma^{-1}k^{-1},g)\sigma(\gamma^{-1}k^{-1}g,k\gamma)\bar{\sigma}(k^{-1},g)\bar\sigma(k^{-1}g,k),
\end{multline*}
which equals $1$ by repeated applications of Lemma \ref{lem invert and swap}.
\end{proof}

With these preparations, we are now ready to prove that $\tr^{(g)}_\sigma$ is tracial and well-defined independently of the choice of functions used in Definition \ref{def g trace}.
\begin{proposition}
\label{prop trace properties}
	Suppose $(g)$ is a $\sigma$-regular conjugacy class with respect to the multiplier $\sigma$ defined by \eqref{eq sigma}. Then:
	\begin{enumerate}[(i)]
	\item $\textnormal{tr}^{(g)}_\sigma$ does not depend on the choices of $c_1$, $c_2$, and $c$ in Definition \ref{def g trace};
	\item if $S$ and $T$ are bounded $(\Gamma,\sigma)$-invariant operators on $L^2(\cS_{\sL})$ such that $ST$ and $TS$ are $(\sigma, g)$-trace class, then $\textnormal{tr}_\sigma^{(g)}(ST)=\textnormal{tr}_\sigma^{(g)}(TS)$.
	\end{enumerate}
\end{proposition}
\begin{proof}
We begin with $(i)$. Let $S$ be a $(\sigma, g)$-trace class operator. By the formula in Lemma \ref{lem integral formula}, it is clear that for any cut-off function $c$, $\tr^{(g)}_\sigma(S)$ is independent of the choice of $c_1$ and $c_2$ such that $c_1 c_2=c$. Define the function $m_1\colon M\to\mathbb{C}$ by
$$m_1(x)\coloneqq\sum_{h\in(g)}\theta(h)\cdot e^{-i\phi_h(x)}\Tr(h^{-1}K_S(hx,x)),$$
so that
\begin{align*}
\textnormal{tr}^{(g)}_\sigma(S)&=\int_{M}c_1(x)m_1(x)c_2(x)\,dx,
\end{align*}
where $c_1,c_2\in C_c(M)$ such that $c_1 c_2=c$ for some cut-off function $c$. We will show that $m_1$ is $\Gamma$-invariant, whence by Lemma \ref{lem useful integration properties}, $\tr^{(g)}$ is independent of $c.$

By $(\Gamma,\sigma)$-invariance in Definition \ref{def kernels}, we see that for any $\gamma\in\Gamma$,
\begin{align}
\label{eq trace computation}
m_1(\gamma x)&=\sum_{h\in(g)}\theta(h)\cdot e^{-i\phi_h(\gamma x)}\Tr(h^{-1}K_S(h\gamma x,\gamma x))\nonumber\\
&=\sum_{h\in(g)}\theta(h)\cdot e^{-i\phi_h(\gamma x)}\Tr(h^{-1}e^{-i\phi_{\gamma^{-1}}(h\gamma x)}\gamma K_S(\gamma^{-1}h\gamma x,x)\gamma^{-1}e^{i\phi_{\gamma^{-1}}(\gamma x)})\nonumber\\
&=\sum_{h\in(g)}\theta(h)\cdot e^{-i\phi_h(\gamma x)}e^{-i\phi_{\gamma^{-1}}(h\gamma x)}\Tr((\gamma^{-1}h\gamma)^{-1}K_S(\gamma^{-1}h\gamma x,x)e^{i\phi_{\gamma^{-1}}(\gamma x)}).
\end{align}
We claim that the $\gamma^{-1}h\gamma$-summand of $m_1(x)$ equals the $h$-summand of $m_1(\gamma x)$. To prove this, it suffices to show that
\begin{equation}
\label{eq smallclaim}
\theta(\gamma^{-1}h\gamma)\cdot e^{-i(\phi_{\gamma^{-1}h\gamma}(x)+\phi_{\gamma^{-1}}(\gamma x))}=\theta(h)\cdot e^{-i(\phi_h(\gamma x)+\phi_{\gamma^{-1}}(h\gamma x))}.
\end{equation}	
Applying identity \eqref{cond 1} with $\gamma\to\gamma^{-1}$, $\gamma'\to\gamma^{-1}h\gamma$, and $x\to\gamma x$, one sees that the function
 $$x\mapsto\phi_{\gamma^{-1}h\gamma}(x)+\phi_{\gamma^{-1}}(\gamma x)-\phi_{\gamma^{-1}h}(\gamma x)$$
 is constant on $M$. Letting $x=x_0$, and using the definition of $\sigma$ in terms of $\phi$ given by \eqref{eq sigma} together with \eqref{cond 2}, we see that
$$\theta(\gamma^{-1} h\gamma)\cdot e^{-i(\phi_{\gamma^{-1}h\gamma}(x_0)+\phi_{\gamma^{-1}}(\gamma x_0)-\phi_{\gamma^{-1}h}(\gamma x_0))}=\theta(\gamma^{-1} h\gamma)\cdot\bar\sigma(\gamma^{-1},\gamma)\sigma(\gamma^{-1}h,\gamma).$$
We then have
 \begin{align*}
 \theta(\gamma^{-1} h\gamma)\cdot\bar\sigma(\gamma^{-1},\gamma)\sigma(\gamma^{-1}h,\gamma)&=\theta(\gamma^{-1} h\gamma)\cdot\bar\sigma(\gamma^{-1}h\gamma,\gamma^{-1})\\
 &=\theta(h)\bar\sigma(\gamma^{-1},h)\\
 &=\theta(h)\bar\sigma(h,\gamma)\bar\sigma(\gamma^{-1},h\gamma)\sigma(\gamma^{-1}h,\gamma)\\
 &=\theta(h)\cdot e^{-i(\phi_h(\gamma x_0)+\phi_{\gamma^{-1}}(h\gamma x_0)-\phi_{\gamma^{-1}h}(\gamma x_0))}.
 \end{align*}
The first equality follows from the cocycle condition. The second follows from Lemma \ref{lem special useful} and the third is again by the cocycle condition. The final equality follows from \eqref{eq sigma}.
 
Again by \eqref{cond 1}, this equals $\theta(h)\cdot e^{-i(\phi_h(\gamma x)+\phi_{\gamma^{-1}}(h\gamma x)-\phi_{\gamma^{-1}h}(\gamma x))}$. From this \eqref{eq smallclaim} follows. Summing over elements of $(g)$ then shows that $m_1$ is $\Gamma$-invariant, which completes the proof of (i).

For (ii), 
note that by Lemma \ref{lem integral formula},	
$$\textnormal{tr}^{(g)}_\sigma(ST)=\sum_{h\in(g)}\theta(h)\cdot \int_{M\times M} e^{-i\phi_h(x)}c(x)\textnormal{Tr}(h^{-1}K_S(hx,y)K_T(y,x))\,dx\,dy,$$
while
\begin{align}
\label{eq trace TS}
	\textnormal{tr}^{(g)}_\sigma(TS)&=\sum_{h\in(g)}\theta(h)\cdot \int_{M\times M} e^{-i\phi_h(y)}c(y)\textnormal{Tr}(h^{-1}K_T(hy,x)K_S(x,y))\,dx\,dy\nonumber\\
	&=\sum_{h\in(g)}\theta(h)\cdot \int_{M\times M} e^{-i\phi_h(h^{-1}y)}c(y)\textnormal{Tr}(h^{-1}K_T(y,x)K_S(x,h^{-1}y))\,dx\,dy,
\end{align}
where we have used the change of variable $y\mapsto h^{-1}y$ and the fact that $y\mapsto c(h^{-1}y)$ is a cut-off function for any $h\in\Gamma$.

Now define $m_2\colon M\times M\to\mathbb{C}$ by
$$m_2(x,y)=\sum_{h\in(g)}\theta(h)\cdot e^{-i\phi_h(x)}\textnormal{Tr}(h^{-1}K_S(hx,y)K_T(y,x))$$
so that
$$\textnormal{tr}^{(g)}_\sigma(ST)=\int_{M\times M}c(x)m_2(x,y)\,dx\,dy.$$
We claim that
\begin{enumerate}[(1)]
\item $m_2$ is $\Gamma$-equivariant for the diagonal action on $M\times M$;
\item $\textnormal{tr}^{(g)}_\sigma(TS)=\displaystyle\int_{M\times M}c(y)m_2(x,y)\,dx\,dy.$
\end{enumerate}
To prove claim (1), fix some $\gamma\in\Gamma$. By the change of variable $h\to\gamma^{-1}h\gamma$, we have
\begin{align*}
m_2(x,y)&=\sum_{\gamma^{-1}h\gamma\in(g)}\theta(\gamma^{-1}h\gamma)\cdot e^{-i\phi_{\gamma^{-1}h\gamma}(x)}\textnormal{Tr}((\gamma^{-1}h\gamma)^{-1}K_S(\gamma^{-1}h\gamma x,y)K_T(y,x)).
\end{align*}
Since $S$ and $T$ are $(\Gamma,\sigma)$-invariant, Lemma \ref{Gamma sigma invariant kernels} implies that $m_2(\gamma x,\gamma y)$ equals
\begin{align*}
&\sum_{h\in (g)}\theta(h)\cdot e^{-i\phi_h(\gamma x)}\textnormal{Tr}(h^{-1}K_S(h\gamma x,\gamma y)K_T(\gamma y,\gamma x))\\
=&\sum_{h\in (g)}\theta(h)\cdot e^{-i\phi_h(\gamma x)}\textnormal{Tr}(e^{-i\phi_{\gamma^{-1}}(h\gamma x)}h^{-1}\gamma K_S(\gamma^{-1}h\gamma x,y)K_T(y,x)\gamma^{-1}e^{i\phi_{\gamma^{-1}}(\gamma x)})\\
=&\sum_{h\in (g)}\theta(h)\cdot e^{-i(\phi_h(\gamma x)+\phi_{\gamma^{-1}}(h\gamma x)-\phi_{\gamma^{-1}}(\gamma x))}\textnormal{Tr}((\gamma^{-1}h\gamma)^{-1}K_S(\gamma^{-1}h\gamma x,y)K_T(y,x)).
\end{align*}
We claim that the $\gamma^{-1}h\gamma$-summand of $m_2(x,y)$ equals the $h$-summand of $m_2(\gamma x,\gamma y)$, that is:
\begin{equation}
\theta(\gamma^{-1}h\gamma)\cdot e^{-i(\phi_{\gamma^{-1}h\gamma}(x))}=\theta(h)\cdot e^{-i(\phi_h(\gamma x)+\phi_{\gamma^{-1}}(h\gamma x)-\phi_{\gamma^{-1}}(\gamma x))};
\end{equation}
but this is precisely what we have already established in \eqref{eq smallclaim}. Now summing over elements of $(g)$ proves claim (1).

Now for claim (2), let us denote the $h$-summand of $m_2(x,y)$ by $[m_2(x,y)]_h$. Then by \eqref{eq trace TS}, it suffices to show that for each $h\in(g)$, we have
\begin{equation}
\label{eq m h-component}
[m_2(x,y)]_h=\theta(h)\cdot e^{-i\phi_h(h^{-1}y)}\textnormal{Tr}(h^{-1}K_T(y,x)K_S(x,h^{-1}y)).
\end{equation}
To see this, note that by Lemma \ref{Gamma sigma invariant kernels}, we have
\begin{align*}
	\textnormal{Tr}(h^{-1}K_S(hx,y)K_T(y,x))&=\textnormal{Tr}(K_T(y,x)h^{-1}K_S(hx,y))\\
	&=\textnormal{Tr}(K_T(y,x)h^{-1}e^{-i\phi_{h^{-1}}(hx)}hK_S(x,h^{-1}y))h^{-1}e^{i\phi_{h^{-1}}(y)})\\
	&=\textnormal{Tr}(h^{-1}K_T(y,x)K_S(x,h^{-1}y))e^{-i(\phi_{h^{-1}}(hx)-\phi_{h^{-1}}(y))}.
\end{align*}
Using the definition of $m_2(x,y)$, this means that
$$[m_2(x,y)]_h=\theta(h)\cdot e^{-i(\phi_h(x)+\phi_{h^{-1}}(hx)-\phi_{h^{-1}}(y))}\textnormal{Tr}(h^{-1}K_T(y,x)K_S(x,h^{-1}y)).$$
Thus to establish \eqref{eq m h-component}, it remains to show that
$$\phi_h(h^{-1}y)+\phi_{h^{-1}}(y)=\phi_h(x)+\phi_{h^{-1}}(hx).$$
By the identity \eqref{cond 1}, the left-hand side equals $\phi_{hh^{-1}}(y)=\phi_e(y)$, while the right-hand side equals $\phi_{hh^{-1}}(hx)=\phi_e(hx)$. Thus equality follows from the normalisation condition $\phi_e\equiv 0$. This completes the proof of claim (2).
 
Finally, since $m_2$ is invariant under the diagonal $\Gamma$-action, we may apply Lemma \ref{lem useful integration properties} (ii) to $m_2$, and use claim (2), to obtain
\begin{align*}
\textnormal{tr}^{(g)}_\sigma(ST)&=\int_{M\times M}c(x)m_2(x,y)\,dx\,dy\\
&=\int_{M\times M}c(y)m_2(x,y)\,dx\,dy\\
&=\textnormal{tr}^{(g)}_\sigma(TS).\qedhere
\end{align*}
\end{proof}

To summarise the notation, we now have:
\begin{itemize}
	\item the trace $\tau^{(g)}_\sigma$ on $\C^\sigma\Gamma$,
	\item the trace $\Tr$ on $\End(\cS_{\sL})$;
	\item the traces $\tr$ and $\tr^{(g)}_\sigma$ on operators on $L^2(\cS_{\sL})$.
\end{itemize}

Recall that for any trace map $\mathscr{T}$ on operators on a Hilbert space $\mathscr{H}$, the associated \emph{supertrace} applied to an operator
\begin{equation}
\label{eqn graded operator}
	S=\begin{pmatrix}S_{11}&S_{12}\\S_{21}&S_{22}\end{pmatrix}
\end{equation}
on the $\Z_2$-graded space $\mathscr{H}\oplus\mathscr{H}$ is
$$
\mathscr{T}(S_{11})-\mathscr{T}(S_{22}).$$
The \emph{supercommutator} of homogeneous elements $S,T\in\cB(\mathscr{H}\oplus\mathscr{H})$ is
$$[S,T]_s\coloneqq ST-(-1)^{\textnormal{deg}\,S\cdot\textnormal{deg}\,T}TS.$$
This extends linearly to arbitrarily elements of $\mathcal{B}(\mathscr{H}\oplus\mathscr{H})$. It follows that the supertrace vanishes on supercommutators.
\begin{definition}
Denote by $\Str$, $\str$, and $\str^{(g)}_\sigma$ the supertraces associated to $\Tr$, $\tr$, and $\tr^{(g)}_\sigma$ respectively.
\end{definition}
Then Definition \ref{def g trace} generalises easily to:
\begin{definition} Let $S$ be a $\mathbb{Z}_2$-graded operator on $L^2(\cS_{\sL})$ written in the form \eqref{eqn graded operator}. If $S_{11}$ and $S_{22}$ are of $(\sigma,g)$-trace class, then we define the \emph{$\sigma$-weighted $(g)$-supertrace} of $S$ to be
\begin{equation}
\label{eq weighted supertrace}
\text{str}_{\sigma}^{(g)}(S)\coloneqq\sum_{h\in(g)}\theta(h)\cdot\text{str}(c_1T_{h^{-1}}Sc_2).
\end{equation}
\end{definition}
\begin{remark}
 	If $S$ has finite propagation, then \eqref{eq weighted supertrace} equals
	$$\sum_{h\in(g)}\theta(h)\cdot\text{str}(cT_{h^{-1}}S),$$
where $c$ is any cut-off function.
\end{remark}

Proposition \ref{prop trace properties} implies:
\begin{corollary}
\label{cor supertrace}
If $S$ and $T$ are $\mathbb{Z}_2$-graded $(\Gamma,\sigma)$-invariant operators on $L^2(\cS_{\sL})$ such that the diagonal entries of $ST$ and $TS$ are of $(\sigma,g)$-trace class, then
$$\str_{\sigma}^{(g)}[S,T]_s=0.$$
\end{corollary}
\begin{remark}
Suppose $g$ is an $\alpha$-regular element, with $\alpha$ as in \eqref{eq alpha}. Then the discussion in this subsection generalises easily to the family of multipliers $\{\sigma^s\}_{s\in\R}$ from \eqref{eq sigma}, once we replace the projective representation $T$ by $T^s$ from Definition \ref{def proj action}, the family $\phi$ by $s\phi$, and $\theta$ by $\theta^s$ from \eqref{eq theta s} in the proof of Corollary \ref{cor trace any s}.
\end{remark}

\vspace{0.1in}
\subsection{Twisted Roe algebras}
 	\hfill\vskip 0.05in

\noindent To formulate the higher index of projectively invariant operators, we will work with certain geometric $C^*$-algebras. The analogous construction in the non-twisted case is well-known; see for instance \cite{WillettYu}. Since the discussion applies to $\sigma^s$ for any $s$, let us fix $s=1$.
	\begin{definition}
	\label{def:suppprop}
		Let $A$ be an operator on $L^2(\cS_{\sL})$.
		\begin{itemize}[leftmargin=0.29in]
			\item $A$ is \textit{$(\Gamma,\sigma)$-equivariant} if 
			$$T_g^* A T_g=A$$ 
			for all $g\in\Gamma$, where $T_g$ is as in Definition \ref{def proj action};
			\item The \emph{support} of $A$, denoted $\textnormal{supp}(A)$, is the complement of all $(x,y)\in M\times M$ for which there exist $f_1,f_2\in C_0(M)$ such that $f_1(x)\neq 0$, $f_2(y)\neq 0$, and
			$$f_1Af_2=0;$$
			\item The \textit{propagation} of $A$ is the extended real number $$\textnormal{prop}(A)=\sup\{d(x,y)\,|\,(x,y)\in\textnormal{supp}(A)\},$$ where $d$ denotes the Riemannian distance on $M$;
			\item $A$ is \textit{locally compact} if $fA$ and $Af\in\mathcal{K}(L^2(\cS_{\sL}))$ for all $f\in C_0(M)$.
		\end{itemize}
		\end{definition}
\begin{definition}
\label{def twisted Roe}
		The \emph{$(\Gamma,\sigma)$-equivariant algebraic Roe algebra of $M$}, denoted by $\mathbb{C}[M;L^2(\cS_{\sL})]^{\Gamma,\sigma}$, is the $*$-subalgebra of $\mathcal{B}(L^2(\cS_{\sL}))$ consisting of $(\Gamma,\sigma)$-equivariant locally compact operators with finite propagation. 

			The \emph{$(\Gamma,\sigma)$-equivariant Roe algebra of $M$}, denoted by $C^*(M;L^2(\cS_{\sL}))^{\Gamma,\sigma}$, is the completion of $\mathbb{C}[M;L^2(\cS_{\sL})]^{\Gamma,\sigma}$ in $\mathcal{B}(L^2(\cS_{\sL}))$.
\end{definition}
\begin{definition}
\label{def kernels}
Consider the vector bundle $\Hom(\cS_{\sL})=\cS_{\sL}\boxtimes\cS_{\sL}\to M\times M$. Let $\mathscr{S}(M)^{\Gamma,\sigma}$ denote the convolution algebra of smooth sections $k$ of $\Hom(\cS_{\sL})$ such that
\begin{enumerate}[label=(\roman*), leftmargin=0.29in]
	\item $k$ is $(\Gamma,\sigma)$-invariant, in the sense that $$e^{-i\phi_\gamma(x)}\gamma^{-1}k_A(\gamma x,\gamma y)\gamma e^{i\phi_\gamma(y)}=k_A(x,y)$$
	for all $\gamma\in\Gamma$ and $x,y\in M$;
	\item $k$ has finite propagation, in the sense that there exists an $R>0$ such that $k(x,y)=0$ whenever $d(x,y)>R$.
\end{enumerate}
\end{definition}

An element $k\in\mathscr{S}(M)^{\Gamma,\sigma}$ acts on a section $u\in L^2(\cS_{\sL})$ by
\begin{equation}
\label{eq kernel formula}
(ku)(x)=\int_Mk(x,y)u(y)\,dy.
\end{equation}

\begin{lemma}
\label{Gamma sigma invariant kernels}
Suppose a $(\Gamma,\sigma)$-invariant operator $A$ on $L^2(\cS_\sL)$ with finite propagation has a smooth Schwartz kernel $k_A$. Then $k_A\in\mathscr{S}(M)^{\Gamma,\sigma}$.
\end{lemma}
\begin{proof}
The finite-propagation property for $k_A$ follows from the fact that $A$ has finite propagation.
The assumption $T_{\gamma}^*AT_{\gamma}=A$, together with the fact that $T_\gamma^*=T_\gamma^{-1}$, implies that
\begin{align*}
k_A(x,y)&=k_{T_{\gamma}^{-1}AT_{\gamma}}(x,y)\\
&=k_{S_{\gamma}^{-1}U_{\gamma}^{-1}AU_{\gamma}S_{\gamma}}(x,y)\\
&=e^{-i\phi_\gamma(x)}k_{U_{\gamma}^{-1}AU_{\gamma}}(x,y)e^{i\phi_\gamma(y)}\\
&=e^{-i\phi_\gamma(x)}\gamma^{-1}k_A(\gamma x,\gamma y)\gamma e^{i\phi_\gamma(y)}.\qedhere
\end{align*}
\end{proof}
Conversely, an element of $\mathscr{S}(M)^{\Gamma,\sigma}$ defines an element of $\mathbb{C}[M;L^2(\cS_{\sL})]^{\Gamma,\sigma}$ via the action \eqref{eq kernel formula}.

		

Let $\mathcal{M}\subseteq\cB(L^2(\cS_{\sL}))$ denote the multiplier algebra of the $(\Gamma,\sigma)$-equivariant Roe algebra, and write $\mathcal{Q}=\mathcal{M}/C^*(M;L^2(\cS_{\sL}))^{\Gamma,\sigma}$. We have a short exact sequence of $C^*$-algebras
	\begin{equation}
	\label{eq SES Roe}
	0\rightarrow C^*(M;L^2(\cS_{\sL}))^{\Gamma,\sigma}\rightarrow\mathcal{M}\rightarrow\mathcal{Q}\rightarrow 0.
	\end{equation}

\vspace{0.1in}
	\subsection{The twisted higher index}
	\label{subsec twisted higher index}
 	\hfill\vskip 0.05in
 	\noindent We discuss the index map first in a more general context. Recall that associated to any short exact sequence of $C^*$-algebras
	$$0\rightarrow I\rightarrow A\rightarrow A/I\rightarrow 0,$$
	is a cyclic exact sequence in $K$-theory:
	\[
	\begin{tikzcd}
	K_0(I) \ar{r} & K_0(A) \ar{r} & K_0(A/I) \ar{d}{\partial_1} \\
	K_1(A/I) \ar{u}{\partial_0} & K_1(A) \ar{l} & K_1(I) \ar{l},
	\end{tikzcd}
	\]
	where the connecting maps $\partial_0$ and $\partial_1$ are defined as follows.
	\begin{definition}
		\label{def connectingmaps}
		\hfill
		\begin{enumerate}[label=(\roman*), leftmargin=0.29in]
			\item $\partial_0$: let $u$ be an invertible matrix with entires in $A/I$ representing a class in $K_1(A/I)$.
			Write
			\[w=
			\begin{pmatrix}
			0&-u^{-1}\\u&0	
			\end{pmatrix}=
			\begin{pmatrix} 
			1 & 0\\ 
			u & 1
			\end{pmatrix}
			\begin{pmatrix} 
			1 & -u^{-1}\\ 
			0 & 1
			\end{pmatrix}
			\begin{pmatrix} 
			1 & 0\\ 
			u & 1
			\end{pmatrix}.
			\]
			Then $w$ lifts to an invertible matrix $W$ with entries in $A$	. Then
			$$
			P=W\begin{pmatrix}
			1 & 0\\ 
			0 & 0
			\end{pmatrix}W^{-1}
			$$
			is an idempotent, and we define
			\begin{equation}
			\label{eq even index}
			\partial_0[u]\coloneqq 
			\left[P
			\right]-
			\begin{bmatrix}
			0 & 0\\
			0 & 1
			\end{bmatrix}
			\in K_0(I).
			\end{equation}

			\item $\partial_1$: let $q$ be an idempotent matrix with entries in $A/I$ representing a class in $K_0(A/I)$. Let $Q$ be a lift of $q$ to a matrix algebra over $A$. Then we define
			\begin{equation}
			\label{eq odd index}
			\partial_1[q]\coloneqq\left[e^{2\pi iQ}\right]\in K_1(I).
			\end{equation}
		\end{enumerate}
	\end{definition}
	
Now let $M$ be a Riemannian spin manifold on which $\Gamma$ acts properly and isometrically, respecting the spin structure. Let $\sigma$ be a multiplier on $\Gamma$. Let $\sL\to M$ be a trivial line bundle, and let $D$ be the twisted Dirac operator as in \eqref{eq twisted Dirac} acting on smooth sections of $\cS_\sL$. Pick any \emph{normalizing function} $\chi\colon\mathbb{R}\rightarrow\mathbb{R}$, i.e. a continuous, odd function such that 
	$$\lim_{x\rightarrow +\infty}\chi(x)=1,$$
and form the bounded self-adjoint operator $\chi(D)$ on $L^2(\cS_{\sL})$. When $\dim M$ is even, $\cS_{\sL}$ is naturally a direct sum $(\cS^+\otimes\sL)\oplus (\cS^-\otimes\sL)$, and $D$ and $\chi(D)$ are odd-graded:
$$D=\begin{pmatrix}0&D_-\\D_+&0\end{pmatrix},\qquad\chi(D)=\begin{pmatrix}0&\chi(D)_-\\\chi(D)_+&0\end{pmatrix}.$$
 
	\begin{proposition}
		The class of $\chi(D)$ in $\mathcal{M}/C^*(M;L^2(\cS_{\sL}))^{\Gamma,\sigma}$ is invertible and independent of the choice of $\chi$, and $\frac{\chi(D)+1}{2}$ is an idempotent modulo $C^*(M;L^2(\cS_{\sL}))^{\Gamma,\sigma}$.
	\end{proposition}
	\begin{proof} 
	Since $\chi^2-1\in C_0(\mathbb{R})$, it suffices to show that for any $f\in C_0(\mathbb{R})$, we have $f(D)\in C^*(M;L^2(\cS_{\sL}))^{\Gamma,\sigma}$. In fact, functions in the Schwartz algebra $\cS(\mathbb{R})$ with compactly supported Fourier transform form a dense subset of $C_0(\mathbb{R})$, we may assume that $f$ is such a function. In that case, the operator $f(D)$ is given by a smooth $(\Gamma,\sigma)$-invariant Schwartz kernel, and hence an element of $\mathscr{S}(M)^{\Gamma,\sigma}$ (see Definition \ref{def kernels}). It follows that $f(D)\in C^*(M;L^2(\cS_{\sL}))^{\Gamma,\sigma}$. Finally, since the difference of any two normalizing functions lies in $C_0(\mathbb{R})$, the class of $\chi(D)$ does not depend on the choice of normalising function $\chi$.
	\end{proof}

Applying Definition \ref{def connectingmaps} to the short exact sequence \eqref{eq SES Roe} leads to the following:
\begin{definition}
\label{def higher index}
			For $i=0,1$, let $\partial_i$ be the connecting maps from Definition \ref{def connectingmaps}. The \emph{$(\Gamma,\sigma)$-invariant higher index of $D$ on $L^2(\cS_{\sL})$} is the element
			\begin{empheq}[left={\Ind_{\Gamma,\sigma,L^2} (D)\coloneqq
\empheqlbrace}]{alignat*=2}
    \partial_{0}\left[\chi(D)_+\right]\in K_{0}\big(C^*(M;L^2(\cS_{\mathscr{L}}))^{\Gamma,\sigma}\big)\quad&\textnormal{ if $\dim M$ is even},\\[1.5ex]
    \partial_{1}\left[\tfrac{\chi(D)+1}{2}\right]\in K_{1}\big(C^*(M;L^2(\cS_{\sL}))^{\Gamma,\sigma}\big)\quad&\textnormal{ if $\dim M$ is odd}.
\end{empheq}
\end{definition}

From this, we can obtain an index in the $K$-theory of $C^*_r(\Gamma,\sigma)$ as follows. Define 
\begin{align}
\label{eq j}
j_\sigma\colon C_c(\cS_{\sL})&\hookrightarrow\mathbb{C}^\sigma[\Gamma]\otimes C_c(\cS_{\sL})\nonumber\\
(j_\sigma e)_\gamma&=c\cdot(T_{g^{-1}}e),
\end{align}
where $T$ is the $(\Gamma,\sigma)$-representation on $L^2(\cS_\sL)$.
Then $j_\sigma$ extends to an isometry
$$j_\sigma\colon L^2(\cS_\sL)\hookrightarrow l^2(\Gamma)\otimes L^2(\cS_\sL).$$
The following is easily proved:
\begin{lemma}
The map $j_\sigma$ is equivariant with respect to the projective representation $T$ on $L^2(\cS_\sL)$ and $T^L\otimes 1$, where $T^L$ is the left-regular $(\Gamma,\sigma)$-representation on $l^2(\Gamma)$, as in subsection \ref{subsec twisted algebras}.
\end{lemma}

At the level of operators, the map $j_\sigma$ induces two inclusion maps
\begin{align*}
	\oplus\,0,\,\oplus\,1\colon\mathcal{B}(L^2(\cS_\sL))&\hookrightarrow\mathcal{B}(l^2(\Gamma)\otimes L^2(\cS_\sL)),\end{align*}
defined by first identifying operators on $L^2(\cS_\sL)$ with operators on $j(L^2(\cS_\sL))$ via conjugation by $j$ and then extending them by the zero operator and identity operator, respectively, on the orthogonal complement of $j(L^2(\cS_\sL))$. We will write $T\oplus i\coloneqq(\oplus\,i)(T)$, $i=0,1$.

These maps restrict to inclusions
\begin{align*}
	\oplus\,0\colon C^*(M;L^2(\cS_\sL))^{\Gamma,\sigma}&\to C^*_r(\Gamma,\sigma)\otimes\mathcal{K}(L^2(\cS_\sL)),\\
	\oplus\,1\colon (C^*(M;L^2(\cS_\sL))^{\Gamma,\sigma})^+&\to (C^*_r(\Gamma,\sigma)\otimes\mathcal{K}(L^2(\cS_\sL)))^+.
\end{align*}
For each $i=0,1$, the map $\oplus \,i$ extends in the obvious way to maps between matrix algebras over $C^*(M;L^2(\cS_\sL))^{\Gamma,\sigma}$ and its unitisation, preserving idempotence and invertibility when $i=0$ and $1$ respectively. We get induced maps
$$(\oplus\,i)_*\colon K_i(C^*(M;L^2(\cS_\sL))^{\Gamma,\sigma})\to K_i(C^*_r(\Gamma,\sigma)\otimes\mathcal{K}(L^2(\cS_\sL)))\cong K_i(C^*_r(\Gamma,\sigma)).$$
This gives an index
\begin{equation}
\label{eq c star g index}
	\Ind_{{\Gamma,\sigma}}(D)\coloneqq(\oplus\,i)_*\Ind_{{\Gamma,\sigma,L^2}}(D)\in K_i(C^*_r(\Gamma,\sigma)),
\end{equation}
where $i=\dim M$ (mod $2$).
\begin{remark}
When $\sigma\equiv 1$, \eqref{eq c star g index} recovers the usual $\Gamma$-equivariant higher index of Dirac operators on cocompact manifolds.
\end{remark}
\hfill\vskip 0.3in
\section{Projective PSC obstructions on cocompact manifolds}
\label{sec cocompact obstr}
In this section we prove Theorem \ref{thm main1}. Thus let $\Gamma$, $M$, $\alpha$, $\sigma$, and $g$ be given as in the statement of the theorem. 


We begin by describing how the trace $\tau^{(g)}_\sigma$ from Definition \ref{def weighted trace} can be extended to a trace on a smooth dense subalgebra of $C^*(\Gamma,\sigma)\otimes\mathcal{\cK}(L^2(\cS_\sL))$. Our construction is based on \cite[section 6]{CM90}, \cite[subsection 2.2]{XY2}, and \cite[section 3]{KMS}, but some extra arguments are needed to make the adaptation to the proper action case, having to do with the form of the embedding $j_\sigma$ in \eqref{eq j}.

Let $D$ and $\cS_\sL$ be as in Definition \ref{def projective Dirac}, with the underlying manifold $M$ being cocompact with respect to the action of a finitely generated group $\Gamma$. We will work with the following choice of orthonormal basis of $L^2(\cS_\sL)$. Let $c$ be the cut-off function for the $\Gamma$-action on $M$ used in the embedding $j$ from \eqref{eq j}. Since the support of $c$ is compact, there exists an orthonormal basis $B_0=\{u_i\}_{i\in\N}$
 of $L^2(\cS_\sL)|_N$, for some relatively compact neighbourhood $N$ of $\text{supp}(c)$ consisting of eigenfunctions of $D^2|_N$. Asymptotically, the eigenvalues $\lambda_i$ satisfy $\lambda_i\sim i^q$ for some constant $q>0$. Identifying $L^2(\cS_\sL)|_N$ with $l^2(\N)$ via this basis, one checks that a smoothing operator supported in $N$ lies in the algebra
 $$\sR=\{(a_{ij})_{i,j\in\N}\colon\sup_{i,j}i^k j^l|a_{ij}|<\infty\textnormal{ for all }k,l\in\N\}$$
of matrices with rapidly decreasing entries \cite[chapter 3]{ConnesNCG}, which are dense inside $\cK(l^2(\N))$. Complete $B_0$ to an orthonormal basis $B=\{v_i\}_{i\in\N}$ of $L^2(\cS_\sL)$, and embed $l^2(\N)$ inside itself via the map $\delta_k\mapsto\delta_{2k}$. Since this map preserves $\sR$, the previous identification now extends to an identification of $L^2(\cS_\sL)$ with the new copy of $l^2(\N)$ in such a way that for any $S\in\sS^{\Gamma,\sigma}$ (see Definition \ref{def kernels}), the operator $c^{\frac{1}{2}}Sc^{\frac{1}{2}}$ is an element of $\sR$. Let
\begin{equation}
\label{eq I}
I=\{i\in\N\colon v_i\in B_0\}.	
\end{equation}

Fix a set of generators of $\Gamma$, and let $\ell(\gamma)=d_\Gamma(\gamma,e)$ denote the associated word length function, where $d_\Gamma$ is the word metric. Let $\cDD$ be the unbounded self-adjoint operator on $l^2(\Gamma)$ given by $\cDD\bar\gamma=\ell(\gamma)\cdot\bar\gamma$, where the notation $\bar\gamma$ is as explained in subsection \ref{subsec twisted algebras}. Let $\Delta$ be the unbounded self-adjoint operator on $l^2(\N)$ given by
\begin{equation}
\label{eq delta}
\Delta(\delta_i)=
\begin{cases}
\lambda_i\delta_i&\textnormal{ if }i\in I,\\
i^\alpha\delta_i&\textnormal{ otherwise,}
\end{cases}
\end{equation}
for $i\in\N$, where $\lambda_i$ is the $i^{\textnormal{th}}$ eigenvalue of $D^2|_N$, and the set $I$ is as in \eqref{eq I}. 

Consider also the unbounded operators $\partial=[\cDD,\,\cdot\,]$ on $\cB(l^2(\N))$ and $\widetilde{\partial}=[\cDD\otimes 1,\,\cdot\,]$ on $\cB(l^2(\Gamma)\otimes l^2(\N))$. Note that $\widetilde{\partial}$ is a closed derivation with domain $\textnormal{Dom}(\widetilde{\partial})$, which consists of all elements $a\in\cB(l^2(\Gamma)\otimes l^2(\N))$ such that $a$ maps $\textnormal{Dom}(\cDD\otimes 1)$ to itself, and the operator $\widetilde{\partial}(a)=(\cDD\otimes 1)\circ a-a\circ(\cDD\otimes 1)$, defined initially on $\textnormal{Dom}(\cDD\otimes 1)$, extends to a bounded operator on $l^2(\Gamma)\otimes l^2(\N)$. Define
$$\sB_\infty(\Gamma,\sigma)=\bigcap_{k\in\N}\textnormal{Dom}(\widetilde{\partial}^k)\cap(C^*_r(\Gamma,\sigma)\otimes\mathcal{K}),$$
and
\begin{equation}
\label{eq B}
\sB(\Gamma,\sigma)=\{a\in\sB_\infty(\Gamma,\sigma)\colon\widetilde{\partial}^k(a)\circ(1\otimes\Delta)^2\textnormal{ is bounded }\forall k\in N\}.
\end{equation}
Note that $\sB(\Gamma,\sigma)$ is a left ideal as well as a Fre\'echet subalgebra of $\sB_\infty(\Gamma,\sigma)$, with respect to the Fr\'echet topology given by the family of norms
$$\norm{a}_n=\sum_{k=0}^\infty\frac{1}{k!}\norm{\widetilde{\partial}^k(a)\circ(1\otimes\Delta)^2},$$
for $n\in\N$, where the norm on the right-hand side is the operator norm.

Define the von Neumann algebra
\begin{equation}
\label{eq sA}
\sA(\Gamma,\sigma)=\{a\in\cB(l^2(\Gamma)\otimes l^2(\N))\colon[\bar\gamma\otimes 1,a]=0\,\,\,\forall\gamma\in\Gamma\}.	
\end{equation}
By \cite[Lemmas 1.1, 3.2]{KMS}, we have the following useful fact:
\begin{lemma}
\label{lem B criterion}
Any $a\in\sA(\Gamma,\sigma)$ can be written as a strongly convergent sum
$$a=\sum_{\gamma\in\Gamma}\bar\gamma\otimes a_\gamma,$$
where $a_\gamma\in\cB(l^2(\N))$ for each $\gamma\in\Gamma$. If $a_\gamma\in\sR$ for each $\gamma\in\Gamma$ and
$$\sum_{\gamma}\ell(\gamma)^k\norm{a_\gamma\circ\Delta}<\infty,$$
for all $k\geq 0$, then $a\in\sB(\Gamma,\sigma)$.
\end{lemma}

\begin{lemma}
\label{lem stable under hfc}
	The $*$-algebra $\sB(\Gamma,\sigma)$ is dense in $C^*_r(\Gamma,\sigma)\otimes\mathcal{K}(l^2(\N))$ and stable under the holomorphic functional calculus.
\end{lemma}
\begin{proof}
First note that if $a=\bar{g}\otimes V\in\C^\sigma\Gamma\otimes\sR$, then
\begin{equation}
\label{eq B contains R}
\widetilde{\partial}^k(A)\circ(1\otimes\Delta)^2=\partial^k(\bar g)\otimes V\Delta^2,
\end{equation}
which is bounded since both $\partial^k(\bar g)$ and $V\Delta^2$ are bounded. It follows that $\sB(\Gamma,\sigma)$ is dense in $C^*_r(\Gamma,\sigma)\otimes\mathcal{K}(l^2(\N))$. 

Now recall that a subalgebra $A$ of a Banach algebra $B$ is called spectral invariant in $B$ if the invertible elements of the unitisation $A^+$ are precisely those which are invertible in $B^+$. In the case that $A$ is a Frechet subalgebra of of $B$, $A$ is spectral invariant if and only if $A$ is stable under the holomorphic functional calculus in $B$, by \cite{Schweitzer}. Further, it is a purely algebraic fact that if $I$ is a left ideal in $A$, then $I$ is itself a spectral invariant subalgebra of $B$.
 
	By \cite[Theorem 1.2]{Ji}, $\sB_\infty(\Gamma,\sigma)$ is dense in $C^*_r(\Gamma,\sigma)\otimes\mathcal{K}(l^2(\N))$ and closed under the holomorphic functional calculus, and hence a spectral invariant subalgebra of $C^*_r(\Gamma,\sigma)\otimes\mathcal{K}(l^2(\N))$. Since $\sB(\Gamma,\sigma)$ is a left ideal in $\sB_\infty(\Gamma,\sigma)$ as well as a Fr\'echet subalgebra, it follows from the above discussion that $\sB(\Gamma,\sigma)$ is holomorphically closed.
\end{proof}

Let $\Tr$ denote the operator trace on $\sR$, and denote by $\tau^{(g)}_\sigma\otimes\Tr$ the amplified trace on $\C\Gamma\otimes\sR$.
Recall that a conjugacy class $(g)$ is said to have \emph{polynomial growth} if there exist constants $C$ and $d$ such that the number of elements $h\in (g)$ such that $\ell(h)\leq l$ is at most $Cl^d$.
\begin{lemma}
\label{lem poly growth}
Let $g\in\Gamma$ be a $\sigma$-regular element with respect to a multiplier $\sigma$. If the conjugacy class $(g)$ has polynomial growth, then $\tau^{(g)}_\sigma\otimes\Tr$ extends to a continuous trace on $\sB(\Gamma,\sigma)$.
\end{lemma}
\begin{proof}
This follows from an adaptation of the proof of \cite[Lemma 2.7]{XY2}. Indeed, let $\mathcal{A}=(\mathcal{A}_{ij})_{i,j\in\N}$ be an arbitrary element of $\sB(\Gamma,\sigma)$, where $\mathcal{A}_{ij}\in C^*_r(\Gamma,\sigma)$ can be written as
\begin{equation}
\label{eq a}
\mathcal{A}_{ij}=\sum_{\gamma\in\Gamma}\mathcal{A}_{ij,\gamma}\bar\gamma.
\end{equation}
Define the function $t\colon\sB(\Gamma,\sigma)\to\C$ by
\begin{equation}
\label{eq t}
t(\mathcal{A})=\sum_{i\in\N}\sum_{h\in(g)}\theta(h)\mathcal{A}_{ii,h}.
\end{equation}
It follows from Definition \ref{def weighted trace} that $t$ agrees with $\tau^{(g)}_\sigma\otimes\Tr$ on $\C\Gamma\otimes\sR$. Now since $\theta(\gamma)\in\U(1)$, the right-hand side of \eqref{eq t} converges absolutely by the same estimates as in the proof of \cite[Lemma 2.7]{XY2}, hence $t$ extends to a continuous trace on $\sB(\Gamma,\sigma)$.
\end{proof}

We will continue to write $\tau^{(g)}_\sigma\otimes\Tr$ for the extended trace $t$ on $\sB(\Gamma,\sigma)$ from the proof of Lemma \ref{lem poly growth}. It follows from and Lemmas \ref{lem stable under hfc} and \ref{lem poly growth} that we have an induced map
\begin{equation}
\label{eq tau induced}
(\tau^{(g)}_\sigma)_*\colon K_0(C^*_r(\Gamma,\sigma)\otimes\cK(L^2(\cS_\sL)))\cong K_0(\sB(\Gamma,\sigma))\to\C.
\end{equation}
We are now ready to establish the connection between $\tau^{(g)}_\sigma$ and the trace $\tr_\sigma^{(g)}$ from section \ref{sec weighted traces}.
\begin{proposition}
\label{prop trace relation}
Let $A$ be a $(\Gamma,\sigma)$-invariant operator on $L^2(\cS_\sL)$ and $g\in\Gamma$ a $\sigma$-regular element with respect to $\sigma$. If $A$ is of $(\sigma,g)$-trace class and $A\oplus 0\in\sB(\Gamma,\sigma)$, then
$$(\tau^{(g)}_\sigma\otimes\Tr)(A\oplus 0)=\tr^{(g)}_{\sigma}(A).$$
\end{proposition}
\begin{proof}
For any $\gamma\in\Gamma$, denote by 
$$q_\gamma\colon l^2(\Gamma)\otimes L^2(\cS_\sL)\to\mathbb{C}\delta_\gamma\otimes L^2(\cS_\sL)$$
the canonical projection. Let $A$ be as given, and write $A\oplus 0\in\sB(\Gamma,\sigma)$ as a matrix $(A_{ij})$, where
\begin{equation}
\label{eq A}
A_{ij}=\sum_{\gamma\in\Gamma}A_{ij,\gamma}\bar\gamma,
\end{equation}
as in \eqref{eq a}. Letting $A_\gamma$ be the element of $\cK$ given by the matrix $(A_{ij,\gamma})_{i,j\in\N}$, one observes that
\begin{equation}
\label{eq qAq}
q_\gamma (A\oplus 0)q_e=\bar\gamma\otimes A_\gamma.	
\end{equation}
Now by \eqref{eq t}, together with the fact that $\tau^{(g)}_\sigma\otimes\Tr$ converges absolutely on elements of $\sB(\Gamma,\sigma)$, we have
\begin{align*}
	(\tau^{(g)}_\sigma\otimes\Tr)(A\oplus 0)&=\sum_{i\in\N}\sum_{h\in(g)}\theta(h)A_{ii,h}\\
	&=\sum_{h\in(g)}\theta(h)\Tr(A_h),
\end{align*}
Thus by Definition \ref{def g trace} it suffices to show that for each $h\in(g)$,
\begin{equation}
\label{eq penultimate}
\Tr(A_h)=\Tr(c^{1/2}T_{h^{-1}}Ac^{1/2}).	
\end{equation}
To this end, let $B=\{u_i\}_{i\in\N}$ be the orthonormal basis of $L^2(\cS_\sL)$ from the start of this section. By definition, 
$$A\oplus 0=j_\sigma Aj_\sigma^*,$$ 
where $j_\sigma$ is as in \eqref{eq j}. Hence by \eqref{eq qAq}, we have
\begin{align}
\label{eq T hat acting on basis}
(\bar h\otimes A_h)(\delta_e\otimes u_i)&=q_h\circ j_\sigma\circ A\circ j_\sigma^*(\delta_e\otimes u_i)
\end{align}
for each $i\in\N$ and $h\in\Gamma$. The map $j_\sigma^*$ can be written explicitly as follows: given $v\in l^2(\Gamma)\otimes L^2(\cS_\sL)$ and $x\in M$,
$$(j_\sigma^*v)(x)=\sum_{\gamma\in\Gamma}c(\gamma^{-1}x)S_{\gamma^{-1}}(\gamma^{-1}x)\cdot \gamma\cdot(v(\gamma,\gamma^{-1}x)).$$
Applying this to $\delta_e\otimes u_i$ and using that $S_e\equiv 1$ gives
\begin{align*}
j_\sigma^*(\delta_e\otimes u_i)(x)&=\sum_{\gamma\in\Gamma}c^{1/2}(\gamma^{-1}x)S_{\gamma^{-1}}(\gamma^{-1}x)\cdot\gamma\cdot((\delta_e(\gamma)u_i(\gamma^{-1}x)))\\
&=c^{1/2}(x)u_i(x).
\end{align*}
Thus \eqref{eq T hat acting on basis} equals $q_h\circ j_\sigma\circ A(cu_i)$. By \eqref{eq j}, $j_\sigma(A(cu_i))(\gamma)=c^{1/2}T_{\gamma^{-1}}(A(c^{1/2}u_i))$, hence
\begin{align}
\label{eq Agamma}
A_\gamma(u_i)=
c^{1/2}T_{\gamma^{-1}}A(c^{1/2}u_i).
\end{align}
It follows that for any $h\in(g)$,
\begin{align*}
\Tr(A_h)&=
\sum_i\langle c^{1/2}T_{h^{-1}}A(c^{1/2}u_i),u_i\rangle_{L^2(\cS_\sL)}\\
&=\Tr(c^{1/2}T_{h^{-1}}Ac^{1/2}),
\end{align*}
which establishes \eqref{eq penultimate}.
\end{proof}

\begin{theorem}
\label{thm trace relation}
	Let $M$, $\Gamma$, and $D$ be as in the statement of Theorem \ref{thm main1} and Definition \ref{def projective Dirac}. Let $g\in\Gamma$ be a $\sigma$-regular element for a multiplier $\sigma$. Then
	\begin{enumerate}[(i)]
	\item $e^{-tD^2}\oplus 0\in\sB(\Gamma,\sigma)$, where $\sB(\Gamma,\sigma)$ is as in \eqref{eq B};
	\item we have $$\big(\tau^{(g)}_\sigma\big)_*\Ind_{\Gamma,\sigma}(D)=\textnormal{str}_{\sigma}^{(g)}(e^{-tD^2}).$$
	\end{enumerate}
\end{theorem}
\begin{proof}
For (i), note that by \eqref{eq qAq} and \eqref{eq Agamma}, we can write
\begin{equation*}
\label{eq heat plus 0}
e^{-tD^2}\oplus 0=\sum_{\gamma\in\Gamma}\bar\gamma\otimes c^{1/2}T_{\gamma^{-1}}e^{-tD^2}c^{1/2}.	
\end{equation*}
Under the identification of $L^2(\cS_\sL)$ with $l^2(\N)$ given at the start of this section, the operator $c^{1/2}T_{\gamma^{-1}}e^{-tD^2}c^{1/2}$ is an element of $\sR$. Hence by Lemma \ref{lem B criterion}, it suffices to prove that 
\begin{equation}
\label{eq convergence in B}
\sum_{\gamma}\ell(\gamma)^k\norm{c^{1/2}T_{\gamma^{-1}}e^{-tD^2}c^{1/2}\circ\Delta}<\infty
\end{equation}
for any $k\geq 0$, where the norm is the operator norm on $l^2(\N)\cong L^2(\cS_\sL)$. 

Note that by construction, the operator $\Delta$ is local in the sense that for any $u\in L^2(\cS_\sL)$, we have $\supp(\Delta u)\subseteq\supp(u)$,
while if $\supp(u)\subseteq\supp(c)$, then $\Delta$ acts on $u$ as $D^2$. Since the operator $c^{1/2}T_{\gamma^{-1}}e^{-tD^2}c^{1/2}$ is supported only on the subset $\supp(c)\times\supp(c)\subseteq M\times M$,
and $c^{1/2}D^2=D^2 c^{1/2}+DA_1+A_2$ for some operators $A_1,A_2$ of order zero, it follows from the Gaussian decay of the heat kernel that there exist constants $C_1$, $C_2$, and $C_3$ such that for all $\gamma$ with $\ell(\gamma)>C$, we have
$$\norm{c^{1/2}T_{\gamma^{-1}}e^{-tD^2}c^{1/2}D^2}\leq C_2 e^{-C_3\ell(\gamma)^2}.$$
Since $\Gamma$ is finitely generated, there exist constants $C_4, C_5$ such that the number of group elements $\gamma$ with $\ell(\gamma)\leq l$ is at most $C_4 e^{C_5 l}$. Setting
$$C_6=\sum_{\ell(\gamma)<C_1}\ell(\gamma)^k\norm{c^{1/2}T_{\gamma^{-1}}e^{-tD^2}c^{1/2}D^2}$$
and combining the above observations shows that
\begin{align*}
\sum_{\gamma}\ell(\gamma)^k\norm{c^{1/2}T_{\gamma^{-1}}e^{-tD^2}c^{1/2}D^2}&=C_6+\sum_{\ell(\gamma)\geq C_1}\ell(\gamma)^k\norm{c^{1/2}T_{\gamma^{-1}}e^{-tD^2}c^{1/2}D^2}\\
	&\leq C_6+\sum_{\ell(\gamma)\geq C_1}\ell(\gamma)^kC_4 e^{C_5\ell(\gamma)}C_2 e^{-C_3\ell(\gamma)^2},
\end{align*}
which is finite. This establishes \eqref{eq convergence in B}.

For (ii), note that we have a commutative diagram
	\[
	\begin{tikzcd}
	K_0(C^*_r(\Gamma,\sigma)\otimes\mathcal{K}(L^2(\cS_\sL)))\ar{r}{\Tr_*}\ar{d}{(\tau^{(g)}_\sigma\otimes\Tr)_*}& K_0(C^*_r(\Gamma,\sigma))\ar{d}{(\tau^{(g)}_\sigma)_*} \\
	\C\ar{r}{=} & \C,
	\end{tikzcd}\qquad\qquad
	\]
	where the map $\Tr_*$ is induced by the operator trace on $L^2(\cS_\sL)$. It follows from \cite[Exercise 12.7.3]{HigsonRoe} that the element
$$(\oplus\,0)_*\Ind_{\Gamma,\sigma,L^2}(D)\in K_0(C^*_r(\Gamma,\sigma)\otimes\mathcal{K}(L^2(\cS_\sL)))$$ 
can be represented explicitly by the following difference of idempotents constructed from the heat operator:
\begin{equation}
\label{eq pchi}
\begin{pmatrix}
			e^{-tD^-D^+}\oplus 0 & e^{-\frac{t}{2}D^-D^+}\frac{1-e^{-tD^-D^+}}{D^-D^+}D^-\oplus 0\\
			e^{-\frac{t}{2}D^+D^-}\frac{1-e^{-tD^+D^-}}{D^+D^-}D^+\oplus 0 & (1-e^{-tD^+D^-})\oplus 0
\end{pmatrix}
-
\begin{pmatrix}0&0\\0&1\end{pmatrix}.
\end{equation}
By part (i), the diagonal entries are in the unitisation of $\sB(\Gamma,\sigma)$, and a similar argument shows that this is also true of the off-diagonal entries. Thus
\begin{align*}
\qquad\quad\big(\tau^{(g)}_\sigma\big)_*(\Ind_{\Gamma,\sigma}(D))&=(\tau^{(g)}_\sigma\otimes\Tr)_*(\oplus\,0)_*\Ind_{\Gamma,\sigma,L^2}(D)\\
&=(\tau^{(g)}_\sigma\otimes\Tr)(e^{-tD^-D^+}\oplus 0)+(\tau^{(g)}_\sigma\otimes\Tr)(-e^{-tD^+D^-}\oplus 0)\\
&=\str^{(g)}_\sigma(e^{-tD^2}).	&\qedhere
\end{align*}
where we have used Proposition \ref{prop trace relation} for the last equality.
\end{proof}

We now turn our attention to Theorem \ref{thm main1}. To begin, we have:
\begin{proposition}
	\label{prop small s}
	Let $M$, $\Gamma$, $D$, and $\sigma$ be as in Theorem \ref{thm main1}. Let $\omega$ be as in \eqref{eq omega}. If the $\Gamma$-invariant Riemannian metric on $M$ satisfies
	$$\inf_{x\in M}(\kappa(x)-4\norm{c(\omega)}_x)>0,$$
	where the norm is taken fibrewise in $\End(\cS_\sL)$, then
	$$\Ind_{\Gamma,\sigma,L^2}(D)=0\in K_{0}\big(C^*(M;L^2(\cS_\sL))^{\Gamma,\sigma}\big).$$
	In particular $\Ind_{\Gamma,\sigma}(D)=0\in K_0(C^*_r(\Gamma,\sigma)).$
\end{proposition}
The proof of this proposition uses the Bochner-Lichnerowicz formula.
\begin{lemma}[Bochner-Lichnerowicz]
\label{lem Bochner}
	Let $N$ be a spin Riemannian manifold, and let $\cS_N\to N$ be the spinor bundle with connection $\nabla^{\cS_N}$. Let $\nabla^E$ be a Hermitian connection on a Hermitian vector bundle $E\to M$. Then
	$$D_E^2=\nabla^*\nabla+\frac{\kappa}{4}+c(R^E),$$
	where $D_E$ is the Dirac operator associated to the connection $\nabla=\nabla^{\cS_N}\otimes 1+1\otimes\nabla^E$ on $\cS_N\otimes E$, $\kappa$ is the scalar curvature on $N$, and $c(R^E)$ denotes Clifford multiplication by the curvature of $\nabla^E$.
\end{lemma}

\begin{proof}[Proof of Proposition \ref{prop small s}]
Let $\nabla^L$ be the Hermitian connection on the trivial bundle $\sL\to M$ defined by $i\eta$, where $\eta$ is any one-form satisfying $d\eta=\omega$. Applying Lemma \ref{lem Bochner} with $D_E=\slashed{\partial}\otimes\nabla^\sL$ gives  
$$(\slashed{\partial}\otimes\nabla^\sL)^2=\nabla^*\nabla+\frac{\kappa}{4}+ic(\omega),$$
where $\nabla=\nabla^{\cS}\otimes 1+1\otimes\nabla^\sL$. Thus for all $u\in L^2(M,\cS_\sL)$, we have
\begin{align*}
\langle	D^2 u,u\rangle_{L^2}&=\langle\nabla u,\nabla u\rangle_{L^2}+\frac{1}{4}\langle\kappa u,u\rangle_{L^2}+\langle ic(\omega)u,u\rangle_{L^2}\geq 0
\end{align*}
by our assumption. Then $\Ind_{\Gamma,\sigma,L^2}(D)$ can be defined by choosing the normalising function $\chi$ equal to the sign function, and a routine computation then shows that the index representative \eqref{eq Aj graded} is exactly zero. The final statement follows by applying the homomorphism $(\oplus\,0)_*$ to $\Ind_{\Gamma,\sigma,L^2}(D)$, as in \eqref{eq c star g index}.
\end{proof}

\begin{corollary}
\label{cor vanishing higher}
	Let $M$ and $\Gamma$ be as in Theorem \ref{thm main1}. For each $s\in\R$, let $D^s$ be the twisted Dirac operator constructed in subsection \ref{subsec multipliers}. If the metric on $M$ has positive scalar curvature, then for $s$ sufficiently small we have $\Ind_{\Gamma,\sigma^s,L^2}(D^s)=0$.
\end{corollary}
\begin{proof}
The curvature of the connection $\nabla^{\sL,s}$ is $is\omega$. Thus by Proposition \ref{prop small s}, it suffices to show that
\begin{equation}
\label{eq pointwise small s}
\inf_{x\in M}(\kappa(x)-4\norm{sc(\omega)}_x)>0	
\end{equation}
for $s$ sufficiently small. For $x\in M$, let $i\lambda_1(x),\ldots,i\lambda_{n}(x)$ be the pointwise eigenvalues of $\omega$, which we view as a skew-symmetric endomorphism of $TM$ using the Riemannian metric. Since the action of $\Gamma$ on $M$ is cocompact, $\kappa(x)$ is uniformly bounded below by some $\kappa_0>0$, while there exists a constant $C$ such that
$$\norm{c(\omega)}_x\leq C\cdot\sum_{j=1}^n|\lambda_j(x)|$$
for each $x$. Since $\omega$ is $\Gamma$-invariant, the right-hand side is again uniformly bounded over $M$. It follows that \eqref{eq pointwise small s} holds for $s$ sufficiently small.
\end{proof}

\begin{proposition}
\label{prop heat trace}
Let $(g)\subseteq\Gamma$ be a $\sigma^s$-regular conjugacy class, with $\sigma^s$ as in \eqref{eq sigma}. Then the operators $e^{-tD^{s}_-D^{s}_+}$ and $e^{-tD^{s}_+D^{s}_-}$ are of $(\sigma^s, g)$-trace class.
\end{proposition}
\begin{proof}
	Fix $s\in\R$ and $\phi_1,\phi_2\in C_c(M)$. Let $K_{s,t}$ be the Schwartz kernel of $e^{-t(D^s)^2}$. By standard estimates for the heat kernel on cocompact manifolds (see for example \cite[Corollary 3.5]{Wang-Wang}), there exists a constant $C_2$, depending on $t$, such that
$$\sum_{h\in(g)}\norm{h^{-1}K_{s,t}(hx,x)}_x<C_2$$
for all $x\in M$, where the norm is taken in $\End(\cS_\sL)$. Then
	\begin{align*}
	\label{eq heat g trace}
	\sum_{h\in(g)}|\theta(h)\cdot\Str(\phi_1 T_{h^{-1}}e^{-t(D^s)^2}\phi_2)|&\leq\sum_{h\in(g)}|\Str(\phi_1 T_{h^{-1}}e^{-t(D^s)^2}\phi_2)|\\
	&\leq C_1\sum_{h\in(g)}\int_M\phi_1(x)\phi_2(x)\norm{h^{-1}K_{s,t}(hx,x)}_x\,dx\\
	&\leq C_1 C_2\|\phi_1\phi_2\|_{L^1}.
	\end{align*}
	for some constant $C_1$,
by Fubini's theorem.
Hence $e^{-t(D^s)^2}$ is of $(\sigma^s,g)$-trace class. 
\end{proof}
Next, let $\Omega\subseteq\Gamma$ be a subset such that
\begin{enumerate}[(i)]
	\item for all $h\in(g)$, there exists some $k\in\Omega$ such that $k^{-1}gk=h$;
	\item if $k_1,k_2\in\Omega$ and $k_1\neq k_2$, then $k_1^{-1}gk_1\neq k_2^{-1}gk_2$.
\end{enumerate}
Let $c$ be a cut-off function for the action of $\Gamma$ on $M$. One checks that the function $c^g\in C_c^\infty(M)$ defined by
\begin{equation}
\label{eq cg}
c^g(x)=\sum_{k\in\Omega}c(k^{-1}x)
\end{equation}
is a cut-off function for the action of $Z^g$ on $M$.
\begin{proposition}
Let $\alpha$ and $\sigma$ be as in \eqref{eq alpha} and \eqref{eq sigma}. Let $(g)\subseteq\Gamma$ be the conjugacy class of an $\alpha$-regular element and $c^g$ as in \eqref{eq cg}. Then
$$\str^{(g)}_{\sigma^s} e^{-t(D^s)^2}=\int_M e^{-is\phi_g(x)}c^g(x)\Str(g^{-1}K_{s,t}(gx,x))\,dx,$$
where $K_{s,t}$ is the Schwartz kernel of $e^{-t(D^s)^2}$.
\end{proposition}
\begin{proof}
We begin by establishing a useful algebraic identity:
\begin{equation}
\label{eq small identity}
\theta^s(k^{-1}gk)e^{-is(\phi_{k^{-1}gk}(k^{-1}x)+\phi_k(k^{-1}gx)-\phi_k(k^{-1}x))}=e^{-is\phi_g(x)},
\end{equation}
where $\theta^s$ is as in \eqref{eq theta s}. 

It suffices to establish this identity for $s=1$. The case of general $s$ follows by replacing $\theta$ and $\phi$ by $\theta^s$ and $s\phi$ respectively, similar to the proof of Corollary \ref{cor trace any s}. To proceed, note that by \eqref{cond 1} applied with $x\to k^{-1}x$, $\gamma\to k^{-1}gk$, and $\gamma'\to k$, the function
$$x\mapsto-\phi_{k^{-1}gk}(k^{-1}x)-\phi_k(k^{-1}gx)+\phi_{gk}(k^{-1}x)$$
is constant on $M$, hence
\begin{equation}
\label{eq small step}
e^{-i(\phi_{k^{-1}gk}(k^{-1}x)+\phi_k(k^{-1}gx))}=e^{-i\phi_{gk}(k^{-1}x)}\cdot C_1
\end{equation}
for some $C_1\in\U(1)$. Letting $x=kx_0$ and using that $\phi_e\equiv 0$ gives
$$C_1=e^{-i\phi_k(k^{-1}gkx_0)}=\bar\sigma(k,k^{-1}gk).$$
Now by the cocycle identity and Lemma \ref{lem invert and swap},
\begin{align*}
\theta(k^{-1}gk)&=\sigma(g^{-1},k)\sigma(k^{-1},g^{-1}k)\\
	&=\sigma(k^{-1},g^{-1})\sigma(k^{-1}g^{-1}k)\\
	&=\bar\sigma(g,k)\sigma(k^{-1}g^{-1},k).
\end{align*}
Using this and \eqref{eq small step}, the claimed equality \eqref{eq small identity} becomes
\begin{equation}
\label{eq small identity 2}
\bar\sigma(g,k)\sigma(k^{-1}g^{-1},k)\bar\sigma(k,k^{-1}gk)e^{-i(\phi_{gk}(k^{-1}x)+\phi_k(k^{-1}x))}=e^{-i\phi_g(x)}.	
\end{equation}
Applying \eqref{cond 1} with $x\to k^{-1}x$, $\gamma\to k$, and $\gamma'\to g$ now shows that the function
$$x\mapsto-\phi_k(k^{-1}x)-\phi_g(x)+\phi_{gk}(k^{-1}x)$$
is constant on $M$, hence
\begin{equation}
\label{eq small step}
e^{-i(\phi_{gk}(k^{-1}x)+\phi_k(k^{-1}x))}=e^{-i\phi_{g}(x)}\cdot C_2
\end{equation}
for some $C_2\in\U(1)$. Letting $x=kx_0$ shows that
$$C_2=e^{i\phi_g(kx_0)}=\sigma(g,k).$$
This, together with the computation 
\begin{align*}
\sigma(k^{-1}g^{-1},k)\bar\sigma(k,k^{-1}gk)=\sigma(k^{-1}g^{-1},k)\sigma(k^{-1}g^{-1}k,k^{-1})=1,
\end{align*}
establishes \eqref{eq small identity 2} and therefore \eqref{eq small identity}.

Next, by Lemma \ref{lem integral formula} and a computation similar to \eqref{eq trace computation}, we can use the set $\Omega$ defined above to write $\str^{(g)}_{\sigma^s}e^{-t(D^s)^2}$ as
\begin{align*}
&\sum_{k\in\Omega}\theta^s(k^{-1}gk)\int_M e^{-is(\phi_{k^{-1}gk}(k^{-1}x)+\phi_k(k^{-1}gx)-\phi_k(k^{-1}x))}c(k^{-1}x)\Str(g^{-1}K_{s,t}(gx,x))\,dx.
\end{align*}
Similarly to the proof of Proposition \ref{prop heat trace}, we can take the sum inside the integral. Then by \eqref{eq small identity} and the definition of $c^g$ given by \eqref{eq cg}, this is equal to
\begin{align*}
	\int_M e^{-is\phi_g(x)}c^g(x)\Str(g^{-1}K_{s,t}(gx,x))\,dx.&\qedhere
\end{align*}
\end{proof}
We can now finish the proof of Theorem \ref{thm main1}.
\begin{proof}[Proof of Theorem \ref{thm main1}]
For (i), first note that
$$(D^s)^2e^{-t(D^s)^2}=\frac{1}{2}[D^s,D^se^{-t(D^s)^2}]_s.$$
It follows from Corollary \ref{cor supertrace} that
\begin{align}
\frac{d(\str^{(g)}_{\sigma^s} e^{-t(D^s)^2})}{dt}&=-\str_{\sigma^s}^{(g)}\big((D^s)^2e^{-t(D^s)^2}\big)=0,
\end{align} 
hence the function $t\mapsto\str^{(g)}_{\sigma^s} e^{-t(D^s)^2}$ is constant in $t>0$. Let $K_{s,t}$ be the Schwartz kernel of the operator $e^{-t(D^s)^2}$. By standard heat kernel estimates,
in the limit $t\to 0$ the integral
$$\int_M e^{-is\phi_g(x)}c^g(x)\Str(g^{-1}K_{s,t}(gx,x))\,dx$$
localises to arbitrarily small neighbourhoods of the fixed-point submanifold $M^g$; see for example \cite[Lemma 4.10]{HochsWang2}. Pick a sufficiently small tubular neighbourhood and identify it with the normal bundle $\cN$ of $M^g$ in $M$. Since the curvature of the twisting bundle $\sL\to M$ is $\Gamma$-invariant, the standard asymptotic expansion of $K_{s,t}$ (see \cite[Theorem 6.11]{BGV}) applies. The same argument as in the compact case (see \cite[Theorem 6.16]{BGV}) then shows that the above integral equals
\begin{equation}
\label{eq twisted A hat}
\int_{M^g}e^{-is\phi_g}c^g\cdot \frac{\widehat A(M^g)\cdot e^{-s\omega/2\pi i}}{\det(1-g e^{-R^{\mathcal{N}}/2\pi i})^{1/2}},
\end{equation}
where we have used that $c^g|_{M^g}$ is a  cut-off function for the action of $Z^g$ on $M^g$.

By Remark \ref{rem fixed points}, $e^{-is\phi_g}(x)$ is constant on each connected component of $M^g$. Further, the support of $c^g$ is compact and thus intersects only finitely many of these connected components, $M^g_1,\ldots,M^g_N$. For each $j=1,\ldots,N$, pick a point $x_j\in M_j$, and let $\cN_j$ be the restriction of $\cN$ to $M^g_j$. By Theorem \ref{thm trace relation} (ii), together with the above discussion, we have
\begin{align*}
	(\tau^{(g)}_{\sigma^s})_*\Ind_{\Gamma,\sigma^s}(D^s)&=\str^{(g)}_{\sigma^s}(e^{-t(D^s)^2})\\
	&=\sum_{j=1}^m\int_{M^g_j}e^{-is\phi_g(x_j)}c^g\cdot\frac{\widehat A(M^g_j)\cdot e^{-s\omega/2\pi i}|_{M^g_j}}{\det(1-g e^{-R^{\mathcal{N}_j}/2\pi i})^{1/2}}\\
	&=\sum_{k=0}^{\dim M/2}\frac{(-s)^k}{k!}\sum_{j=1}^m\int_{M^g_j}c^g\cdot\frac{\widehat A(M^g)\cdot\left(i\phi_g(x_j)+\frac{\omega}{2\pi i}\right)^k|_{M^g_j}}{\det(1-g e^{-R^{\mathcal{N}_j}/2\pi i})^{1/2}}.
\end{align*}

For (ii), observe that we can, by a straightforward suspension argument, reduce to the case where $M$ is even-dimensional. If $M$ admits a $\Gamma$-invariant metric of positive scalar curvature, then $\Ind_{\Gamma,\sigma^s}D^s$ vanishes for all $s\in (0,\delta)$ for some $0<\delta$, by Proposition \ref{cor vanishing higher}. Since the map
$$s\mapsto(\tau^{(g)}_{\sigma^s})_*(\Ind_{\Gamma,\sigma^s}D^s)$$
is a polynomial in $s$, it must vanish identically on $\mathbb{R}$, hence
\begin{equation}
\label{eq many components vanishing}
\sum_{j=1}^m\int_{M^g_j}c^g\cdot\frac{\widehat A(M^g)\cdot\left(\phi_g(x_j)-\frac{\omega}{2\pi}\right)^k|_{M^g_j}}{\det(1-g e^{-R^{\mathcal{N}_j}/2\pi i})^{1/2}}=0
\end{equation}
for each $k\geq 0$, where we may also assume that each $M^g_j$ is even-dimensional. 

In the case that $M^g$ is connected, we may choose the point $x_0$ from \eqref{cond 2} to lie in $M^g$, whence Remark \ref{rem fixed points} implies that $e^{-is\phi_g}\equiv 1$ on $M^g$.
The formula \eqref{eq many components vanishing} now simplifies to
$$\widehat{A}_g(M,\omega^k)=\int_{M^g}c^g\cdot\frac{\widehat A(M^g)\cdot\omega^k|_{M^g}}{\det(1-g e^{-R^{\mathcal{N}}/2\pi i})^{1/2}}=0$$
for each $k\geq 0$.
\end{proof}
\begin{remark}
\label{rem Moscovici-Wu}
It should be possible to derive part (ii) of Theorem \ref{thm main1} independently of any growth assumptions on $(g)$ by using methods similar to the proof of \cite[Theorem 3.4]{Moscovici-Wu}; see also \cite[Remark A.2]{XY2}. In particular, this would imply that if the Baum-Connes conjecture holds for $\Gamma$, then the map $(\tau^{(g)}_\sigma)_*$ from \eqref{eq tau induced} can be defined via the $\sigma$-weighted $(g)$-supertrace $\str^{(g)}_\sigma$ from \eqref{eq weighted supertrace}.

\end{remark}
\begin{proof}[Proof of Corollary \ref{cor}]
	By a suspension argument, we may assume without loss of generality that both $M$ and $M^g$ are even-dimensional. Further, since the vanishing property is preserved under sums of forms, we may to restrict our attention to the case where
	$$\omega=\omega_1\omega_2\ldots\omega_m,$$
	for some $m\leq\dim M/2$, where each $\omega_i$ is the lift of a differential form on $M/\Gamma$ representing a class in $f^*H^2(\underline B\Gamma,\R)$. For each $i=1,\ldots,m$, let $s_i\in\R$. Then applying to argument from the proof of Theorem \ref{thm main1} to $\sum_{i=1}^m s_i\omega_i$ instead of $s\omega$, with the twisted Dirac operator defined accordingly, shows that if $M$ admits a $\Gamma$-invariant metric of positive scalar curvature, then there exists a $\delta>0$ such that
	$$\sum_{k=0}^{\dim M/2}\frac{1}{k!}\int_{M^g}c^g\cdot\frac{\widehat A(M^g)\cdot\left(\frac{s_1\omega_1+\cdots+s_m\omega_m}{2\pi i}\right)^k|_{M^g}}{\det(1-g e^{-R^{\mathcal{N}}/2\pi i})^{1/2}}=0$$
	whenever $s_i\in(0,\delta)$ for all $i$. Since the left-hand side is a polynomial in the variables $s_1,\ldots,s_m$, it vanishes identically on $\mathbb{R}^m$. In particular, the coefficient of $s_1 s_2\ldots s_m$ is zero, and this is equal to 
	$$\frac{(2\pi i)^{-m}}{m!}\int_{M^g}c^g\cdot\frac{\widehat A(M^g)\cdot\omega|_{M^g}}{\det(1-g e^{-R^{\mathcal{N}}/2\pi i})^{1/2}}.$$
	This concludes the proof.	
\end{proof}
\subsection{Special cases}
\label{subsec special cases}
\hfill\vskip 0.05in
\noindent We discuss two special cases of Theorem \ref{thm main1} that have already appeared elsewhere in the literature. These correspond to the extreme cases when either the conjugacy class or the multiplier is trivial.
\subsubsection{Free action and $(g)=(e)$}
\hfill\vskip 0.05in
\noindent In this case, the canonical trace
\begin{align}
\label{eq tau e}
	\tau_\sigma^{(e)}\colon\mathbb{C}^\sigma\Gamma&\to\mathbb{C}\nonumber\\
	\sum_{\gamma\in\Gamma}a_\gamma\bar\gamma&\mapsto a_e
\end{align}
extends continuously to a trace $C^*_r(\Gamma)\to\C$ and induces a linear map 
$$\big(\tau_\sigma^{(e)}\big)_*\colon K_0(C^*_r(\Gamma,\sigma))\to\C.$$
Since the action of $\Gamma$ on $M$ is free, we may work with the classifying space $B\Gamma$ instead of $\underline{B}\Gamma$, as done in \cite{MathaiTwisted}. Let $f\colon M/\Gamma\to B\Gamma$ be the classifying map for $M$, and let $[\beta]\in H^2(B\Gamma,\R)$. Let $\omega_0$ be a differential form on the quotient manifold $M/\Gamma$ such that $f^*[\beta]=[\omega_0]$, and let $\omega$ be the $\Gamma$-invariant lift of $\omega_0$ to $M$.

Since $\phi_e\equiv 0$, Theorem \ref{thm main1} (i) reduces to the twisted $L^2$-index theorem in of Mathai \cite[Theorem 3.6]{MathaiTwisted} for the case of the spin-Dirac operator, namely that
$$\big(\tau^{(e)}_\sigma\big)_*\Ind_{\Gamma,\sigma}(D)=\int_{M}c\cdot\widehat A(M)\cdot e^{-\omega/2\pi i}=\int_{M/\Gamma}\widehat{A}(M/\Gamma)\cdot e^{-\omega_0/2\pi i},$$
where $c$ is a cut-off function for the $\Gamma$-action on $M$. Theorem \ref{thm main1} (ii) recovers the result if the $M/\Gamma$ admits a metric of positive scalar curvature, then for each non-negative integer $k$,
	$$\int_{M}\widehat{A}(M/\Gamma)\cdot\omega_0^k=0.$$
	This is \cite[Corollary 1]{MathaiTwisted}.

\subsubsection{Trivial multiplier $\sigma\equiv 1$}
\hfill\vskip 0.05in
\noindent In this case, the projectively invariant Dirac operator $D$ is simply the $\Gamma$-invariant Dirac operator acting on sections of the spinor bundle $\cS\to M$. The element $\Ind_{\Gamma,\sigma}(D)$ reduces to the usual $\Gamma$-equivariant higher index $\Ind_{\Gamma}(D)\in K_0(C^*_r(\Gamma))$ for cocompact actions \cite{BCH}. The trace $\tau^{(g)}_\sigma$ is then the unweighted trace $\tau^{(g)}\colon\C\Gamma\to\C$ from \eqref{eq tau g}.

Theorem \ref{thm main1} (i) then states that if $(g)$ has polynomial growth, then
$$\tau^{(g)}_*\Ind_\Gamma(D)=\int_{M^g}c^g\cdot\frac{\widehat A(M^g)}{\det(1-g e^{-R^{\mathcal{N}}/2\pi i})^{1/2}},$$
where we have implicitly taken $\omega$ to be zero. This recovers the formula \cite[Theorem 6.1]{Wang-Wang} in the case of the spin-Dirac operator.

\hfill\vskip 0.3in
\section{A neighbourhood PSC obstruction}
\label{sec Callias obstr}
In this section we prove Theorem \ref{thm main2}. We will make use of a Callias-type index theorem for projectively invariant operators.

\subsection{Projectively invariant Callias-type operators}
\label{subsection proj Callias}
\hfill\vskip 0.05in
\noindent Let us begin with a general definition and discussion of projectively invariant Callias-type operators and their higher indices.

Let $M$ be a complete Riemannian manifold on which a discrete gorup $\Gamma$ acts properly and isometrically, and suppose that $H^1(M)=0$. Let 
$$f\colon M/\Gamma\to\underline{B}\Gamma$$
be the classifying map of $M$. Let $[\beta]\in H^2(\underline{B}\Gamma,\mathbb{R})$, and let $[\omega_0]=f^*[\beta]$ in the de Rham cohomology of the orbifold $M/\Gamma$. Let $\omega$ be the $\Gamma$-invariant lift of $\omega_0$ to $M$. As in subsection \ref{subsec multipliers}, we obtain a one-form $\eta$, a family $\phi$ of functions on $M$, and a family of multipliers $\sigma^s$ on $\Gamma$ parameterised by $s\in\R$.
\begin{remark}
\label{rem alternative sigma}
The cocycles $\alpha$ and $\sigma^s$ from \eqref{eq alpha} and \eqref{eq sigma} can be defined equivalently in terms of the above data restricted to any submanifold $P\subseteq M$ preserved by the action of $\Gamma$. Indeed, by working with the de Rham differential on $P$ instead of $M$, \eqref{cond 1} implies that
the family
$$\phi|_P\coloneqq\{\phi_\gamma|_P\colon\gamma\in\Gamma\}$$
satisfies
\begin{equation*}
d_P(\phi_\gamma|_P+\gamma^{-1}\phi_{\gamma'}|_P-\phi_{\gamma'\gamma}|_P)=0.
\end{equation*}
It follows that $\alpha$ and $\sigma^s$ can be defined equivalently as
$$\alpha(\gamma,\gamma')=\frac{1}{2\pi}(\phi_\gamma|_P(x)+\phi_\gamma|_P(\gamma'x)-\phi_{\gamma'\gamma}|_P(x)),$$
$$\sigma^s(\gamma,\gamma')=e^{2\pi is\alpha(\gamma,\sigma)}.$$


\end{remark}

Since the discussion in rest of this subsection applies uniformly to $\sigma^s$ for any $s$ with only minor and obvious adjustments, let us now fix $s=1$ and write $\sigma=\sigma^1$.

Let $E\to M$ be a $\mathbb{Z}_2$-graded Clifford bundle over $M$ such that $L^2(E)$ is equipped with the projective action $T$ from Definition \ref{def proj action}. Let $D$ be an odd-graded Dirac operator acting on smooth sections of $E$ that commutes with $T$.

\begin{definition}
\label{def Callias}
An odd-graded, $\Gamma$-equivariant fibrewise Hermitian bundle endomorphism $\Phi$ of $E$ is \emph{admissible for $D$} if $D\Phi+\Phi D$ is an endomorphism of $E$ such that there exists a cocompact subset $Z\subseteq M$ and a constant $C>0$ such that the pointwise estimate
\begin{equation}
\label{eq Callias estimate}
\Phi^2\geq\norm{D\Phi+\Phi D}+C
\end{equation}
holds over $M\setminus Z$. In this setting, $D+\Phi$ is called a \emph{$(\Gamma,\sigma)$-invariant Callias-type operator}.
\end{definition}
\begin{remark}
\label{rem projective endomorphism}
A bundle endomorphism on $L^2(E)$ commutes with the projective action $T$ from Definition \ref{def proj action} if and only if it commutes with the unitary action $U$ of the group.
\end{remark}

A key property of the operator $D+\Phi$ is that it has an index in $K_*(C^*_r(\Gamma,\sigma))$. The indices of such operators can be defined either via Roe algebras, similar to what was done in subsection \ref{subsec twisted higher index}, or using Hilbert $C^*_r(\Gamma,\sigma)$-modules. We will take the latter approach in order to frame the discussion in parallel with that in \cite{GHM3} for the untwisted case.

To this end, given elements $a=\sum_{\gamma\in\Gamma}a_\gamma\gamma\in\mathbb{C}^\sigma\Gamma$ and sections $s,s_1,s_2\in C_c(E)$, the formulas
\begin{align}
\label{eq Hilbert modules structure}
	(s\cdot a)(x)&\coloneqq\sum_{\gamma\in\Gamma}a_{\gamma} (T_{\gamma^{-1}} s)(x)\nonumber\\
(s_1,s_2)(\gamma)&\coloneqq(s_1,T_\gamma s_2)
\end{align}
define a pre-Hilbert $\C^\sigma\Gamma$-module structure on $C_c(E)$.
\begin{definition}
\label{def E sigma}
Let $\mathcal{E}^\sigma$ be the Hilbert $C^*_r(\Gamma,\sigma)$-module completion of $C_c(E)$ with respect to \eqref{eq Hilbert modules structure}.
	\end{definition}


The admissibility condition \eqref{eq Callias estimate} implies that the operator $D+\Phi$ is projectively Fredholm in the following sense:
\begin{proposition} \label{thm Guo}
There exists a cocompactly supported, $G$-invariant continuous function $f$ on $M$ such that 
\beq{eq def F}
F\coloneqq  (D+ \Phi)\bigl(  (D + \Phi)^2 + f \bigr)^{-1/2}\in\cB(\cE^\sigma),
\eeq
such that the pair $(\cE^\sigma, F)$ is a cycle in $\KK(\C, C^*_r(\Gamma,\sigma))$. The class $[\cE^\sigma,F]$ is independent of the choice of $f$.
\end{proposition}
\begin{proof}
The proof is analogous to that of \cite[Theorem 4.19]{Guo}. Instead of the Hilbert module $C^*(G)$-module $\mathcal{E}$ used there, we work with $\cE^\sigma$.	
\end{proof}

\begin{definition} 
\label{def index}
The \emph{$(\Gamma,\sigma)$-index} of $D+\Phi$ is the class
\[
\Ind_{\Gamma,\sigma}(D + \Phi) \coloneqq  [\cE^\sigma, F] \in K_0(C^*_r(\Gamma,\sigma))\cong\KK(\C,C^*_r(\Gamma,\sigma)).
\]
\end{definition}
\hfill\vskip 0.05in
\subsection{Localisation of projective Callias-type indices}
\label{subsection Callias localisation}
\hfill\vskip 0.05in
\noindent One of the key properties of Callias-type operators in the equivariant setting is that the their indices can be calculated by localising to a cocompact subset of the manifold \cite[Theorem 3.4]{GHM3}. In the projective setting, a similar result holds. For this, we will assume that the $(\Gamma,\sigma)$-invariant Dirac operator $D$ from Definition \ref{def Callias} takes the form of a Dirac operator twisted by a line bundle, as in Definition \ref{def projective Dirac}.

Let $E_0$ be an ungraded $\Gamma$-equivariant Clifford bundle over $M$ equipped with a $\Gamma$-invariant Hermitian connection $\nabla^{E_0}$. Define the Hermitian connection
\begin{equation}
\label{eq nabla L}	
\nabla^{\sL}=d+i\eta
\end{equation}
on a $\Gamma$-equivariantly trivial Hermitian line bundle $\sL\to M$, and form the connection $\nabla^{E_{0,\sL}}=\nabla^{E_0}\otimes 1+1\otimes\nabla^{\sL}$
	on the bundle
	\begin{equation}
	\label{eq E0L}
	E_{0,\sL}\coloneqq E_0\otimes\sL.
	\end{equation}
 In the notation of subsection \ref{subsection proj Callias}, we will take
	\begin{equation*}
	\label{eq EL}
		E=E_\sL\coloneqq E_{0,\sL} \oplus E_{0,\sL}
	\end{equation*}
	where the first copy of $E_{0,\sL}$ is given the even grading, and the second copy the odd grading. Let $D_0$ be the Dirac operator on $E_{0,\sL}$ associated to $\nabla^{E_{0,\sL}}$, and define 
\beq{eq D D0}
D = 
\begin{pmatrix}
0 & D_0 \\ D_0  & 0
\end{pmatrix}
\eeq
on $E_\sL$. Let $\Phi_0$ be $\Gamma$-invariant a Hermitian endomorphism of $E_0$ such that $\Phi^2\geq\norm{D\Phi+\Phi D}+C$ holds outside a cocompact subset $Z\subseteq M$ for some $C>0$. Let
\beq{eq Phi Phi0}
\Phi = \begin{pmatrix}
0 & i\Phi_0\otimes 1 \\ -i\Phi_0\otimes 1  & 0
\end{pmatrix}.
\eeq
Then $D+\Phi$ is a $(\Gamma,\sigma)$-invariant Callias-type operator acting on sections of $E_\sL$, in the sense of Definition \ref{def Callias}.

Let $M_-\subseteq M$ be a $\Gamma$-invariant, cocompact subset containing $Z$ in its interior, such that $N \coloneqq  \partial M_-$ is a (not necessary connected) smooth submanifold of $M$. Let $M_+$ be the closure of the complement of $M_-$, so that $N = M_- \cap M_+$ and $M = M_- \cup M_+$.  We will use the notation
\[
M = M_- \cup_N M_+.
\]
Let $E^N_{0,\sL}$ denote the restriction of $E_{0,\sL}$ to $N$, equipped with the restricted connection $\nabla^{E_{0,\sL}^N}$. By \eqref{eq Callias estimate}, the restriction of $\Phi_0$ to $N$ is fibrewise invertible. Let 
$$E^N_{0,\sL,\pm}\subseteq E^N_{0,\sL}$$ 
be the positive and negative eigenbundles of $\Phi_0$. Clifford multiplication by $i$ times the unit normal vector field $\widehat n$ to $N$ pointing into $M_+$ defines $\Gamma$-invariant gradings on both $E^N_{0,\sL,+}$ and $E^N_{0,\sL,-}$. Define the connections
\begin{equation}
\label{eq connections}
\nabla^{E^N_{0,\sL,\pm}}\coloneqq p_\pm\nabla^{E_{0,\sL}^N}p_\pm
\end{equation}
on $E^{N}_{0,\sL,\pm}$, where $p_\pm\colon E^N_{0,\sL}\to E^N_{0,\sL,\pm}$ are the orthogonal projections. Along with the Clifford action of $TM|_N$ on $E^{N}_{0,\sL,\pm}$, these connections give rise to two Dirac operators 
\begin{equation}
\label{eq Dirac operators}
D^{E^{N}_{0,\sL,+}}\quad\textnormal{and}\quad D^{E^{N}_{0,\sL,-}},
\end{equation}
both odd-graded, acting on sections of $E^{N}_{0,\sL,+}$ and $E^{N}_{0,\sL,-}$ respectively. 

For each group element $\gamma\in\Gamma$, define unitary operators $U^{N,\pm}_\gamma, S^{N,\pm}_\gamma$, and $T^{N,\pm}_\gamma$ on $L^2(E^N_{0,\sL,\pm})$ by:
\begin{itemize}
\item $U^{N,\pm}_\gamma u(x)=\gamma u(\gamma^{-1}x)$;
\item $S_\gamma^{N,\pm} u=e^{i\phi_\gamma|_N}u$;
\item $T_\gamma^{N,\pm}=U_\gamma^N\circ S_\gamma^{N,\pm}$,
\end{itemize}
where $u\in L^2(E^N_{0,\sL,\pm})$ and $x\in N$. Note that \eqref{eq omega} and \eqref{eq derivative of phi} continue to hold if we restrict $\omega$, $\eta$, and $\phi$ to $N$ and work with the de Rham differential on $N$ instead of $M$. It follows that the operator $D^{E^{N}_{0,\sL,+}}$ is equivariant with respect the projective action $T^{N,+}$. Since $\Gamma$ acts on $N$ cocompactly,  $D^{E^{N}_{0,\sL,+}}$ has a $(\Gamma,\sigma)$-invariant higher index
\[
\Ind_{\Gamma,\sigma}\big(D^{E^{N}_{0,\sL,+}}\big)\in K_0(C^*_r(\Gamma,\sigma))
\]
by Definition \ref{def higher index} and \eqref{eq c star g index}. Equivalently, this index can be formulated as in \eqref{eq def F}, using a bounded transform on a Hilbert module. As mentioned previously, we  will adopt the latter in order to follow more closely the exposition of \cite{GHM3}.

With these preparations, we have the following:
\begin{theorem}[$(\Gamma,\sigma)$-Callias-type index theorem]
\label{thm Callias}
\beq{eq index}
\Ind_{\Gamma,\sigma}(D + \Phi) = \Ind_{\Gamma,\sigma}(D^{E^{N}_{0,\sL,+}}
) \in K_0(C^*_r(\Gamma,\sigma)).
\eeq
\end{theorem}
This is the projective analogue of the equivariant Callias-type index theorem \cite[Theorem 3.4]{GHM3}. The proof is analogous to that in the untwisted setting, once we make the following modifications:	\begin{enumerate}[(1)]
	\item instead of the Hilbert $C^*(G)$-modules used in \cite{GHM3}, we work with Hilbert $C^*_r(\Gamma,\sigma)$-modules;
	\item instead of the $G$-equivariant bundle $S$ used in \cite{GHM3} (resp. bundles derived from $S$), we work with the bundle $E_\sL$ from subsection \ref{subsection Callias localisation} (or bundles derived from it);
	\item instead of $G$-invariant differential operators, we work with $(\Gamma,\sigma)$-invariant operators, as constructed above;
	\item instead of the index $\ind_G$ from \cite[section 3]{GHM3}, we work with $\Ind_{\Gamma,\sigma}$.
	\end{enumerate}
	
Given these similarities, we will for the most part only sketch the proof of Theorem \ref{thm Callias}, and invite the reader who is interested in a more detailed discussion to \cite[section 5]{GHM3}. Nevertheless, let us give a detailed example of how one of the key technical tools used in the proof of \cite[Theorem 3.4]{GHM3} can be adapted to the projective setting, namely the relative index theorem for Callias-type operators \cite[Theorem 4.13]{GHM3}. The twisted analogue of that theorem is as follows. 

For $j=1,2$, let $M_j$, $E_j$, $D_j$ and $\Phi_j$ be as $M$, $E$, $D$, and $\Phi$ were in Definition \ref{def Callias}. Suppose there exist $\Gamma$-invariant, cocompact hypersurfaces $N_j\subseteq M_j$, $\Gamma$-invariant tubular neighbourhoods $U_j\supseteq N_j$, and a $\Gamma$-equivariant isometry $\psi\colon U_1\to U_2$ such that
\begin{itemize}
\item 
 $\psi(N_1) = N_2$;
 \item  $\psi^*(E_2|_{U_2}) \cong E_1|_{U_1}$;
 \item $\psi^*(\nabla_2|_{U_2}) = \nabla_1|_{U_1}$, where $\nabla_j$ is the Clifford connection used to define $D_j$;
 \item  $\Phi_1|_{U_1}$ corresponds to $\Phi_2|_{U_2}$ via $\psi$.
\end{itemize}

Suppose that $M_j = X_j \cup_{N_j} Y_j$ for closed, $\Gamma$-invariant subsets $X_j, Y_j \subseteq M_j$. We identify $N_1$ and $N_2$ via $\psi$ and simply write $N$ for this manifold. Construct
\[
M_3 \coloneqq  X_1  \cup_N Y_2; \qquad M_4 \coloneqq  X_2 \cup_N Y_1.
\]
For $j=3,4$, let $E_j$, $D_j$ and $\Phi_j$  be obtained from the corresponding data on $M_1$ and $M_2$ by cutting and gluing along $U_1 \cong U_2$ via $\psi$. For $j=1,2,3,4$, form the Hilbert $C^*_r(\Gamma,\sigma)$-modules $\cE^\sigma_j$ as in Definition \ref{def E sigma}.

\begin{proposition}
\label{thm 1234}
In the above situation,
\[
\Ind_{\Gamma,\sigma}(D_1 + \Phi_1) + \Ind_{\Gamma,\sigma}(D_2 + \Phi_2) = \Ind_{\Gamma,\sigma}(D_3 + \Phi_3) +  \Ind_{\Gamma,\sigma}(D_4 + \Phi_4)\in K_0(C^*_r(\Gamma,\sigma)).
\]
\end{proposition}
\begin{proof}
(Compare the proof of \cite[Theorem 4.13]{GHM3}.) Define
$$\cE^\sigma\coloneqq\cE^\sigma_1\oplus\cE^\sigma_2 \oplus \cE_3^{\sigma,\op}\oplus\cE_4^{\sigma,\op},$$
where a superscript $\op$ indicates reversal of the $\mathbb{Z}_2$-grading on the given module. Similar to \eqref{eq def F}, define 
$$F_j\coloneqq  (D_j+ \Phi_j)\bigl(  (D_j + \Phi_j)^2 + f_j \bigr)^{-1/2},$$
for $j=1,2,3,4$, and
$$F \coloneqq  F_1 \oplus F_2 \oplus F_3 \oplus F_4.$$
For $j=1,2$, let $\chi_{X_j}, \chi_{Y_j} \in C^{\infty}(M_j)$ be real-valued functions such that:
\begin{enumerate}[(i)]
\item $\supp(\chi_{X_j})\subseteq X_j \cup U_j$ and $\supp(\chi_{Y_j}) \subseteq Y_j \cup U_j$;
\item $\psi^*(\chi_{X_2}|_{U_2}) = \chi_{X_1}|_{U_1}$ and $\psi^*(\chi_{Y_2}|_{U_2}) = \chi_{Y_1}|_{U_1}$; and
\item $\chi_{X_j}^2 + \chi_{Y_j}^2 = 1$.
\end{enumerate}
We view pointwise multiplication by these functions as operators
\beq{eq chi XYj}
\begin{split}
\chi_{X_1}\colon &\cE_1^\sigma \to \cE_3^\sigma;\\
\chi_{Y_1}\colon &\cE_1^\sigma \to \cE_4^\sigma;
\end{split}
\qquad
\begin{split}
\chi_{Y_2}\colon & \cE_2^\sigma \to \cE_3^\sigma;\\
\chi_{X_2}\colon & \cE_2^\sigma \to \cE_4^\sigma.
\end{split}
\eeq
Define the operator
\[
X \coloneqq  \gamma \begin{pmatrix}
0 & 0 & -\chi_{X_1}^* & -\chi_{Y_1}^* \\
0 & 0 & - \chi_{Y_2}^* & \chi_{X_2}^* \\
\chi_{X_1} & \chi_{Y_2} & 0 & 0 \\
\chi_{Y_1} & -\chi_{X_2} & 0 & 0
\end{pmatrix}\in\mathcal{B}(\cE^\sigma),
\]
where $\gamma$ is the grading operator on $\cE^\sigma$. Then $X$ is an odd, self-adjoint operator on $\cE^\sigma$. Further, using properties (ii) and (iii), one verifies directly that $X^2 = 1.$
Let $\mathbb{C}l$ denote the Clifford algebra generated by $X$. It follows from a discussion analogous to \cite[section 4]{GHM3} that
$$XF+FX\in\mathcal{K}(\cE^\sigma).$$

Since $X$ generates $\mathbb{C}l$ and anticommutes with $F$ modulo $\mathcal{K}(\mathcal{E}^\sigma)$, the pair $(\cE^\sigma, F)$ is a Kasparov $(\mathbb{C}l, C_r^*(\Gamma,\sigma))$-cycle. Its class is mapped to $[\cE^\sigma, F] \in K_*(C^*_r(\Gamma,\sigma))$ by the homomorphism induced by the pullback along the inclusion $\C \hookrightarrow \C l$. By \cite[Lemma 1.15]{Bunke}, that homomorphism is zero. Hence
\begin{equation*}
[\cE^\sigma, F] = 0 \in K_0(C^*_r(\Gamma,\sigma)),
\end{equation*}
which is equivalent to the theorem.
\end{proof}

\begin{proof}[Proof of Theorem \ref{thm Callias}]
The first and most important step is to use Proposition \ref{thm 1234} to reduce the computation of $\Ind_{\Gamma,\sigma}(D+\Phi)$ to the index of a $(\Gamma,\sigma)$-invariant Callias-type operator on the cylinder $N\times\mathbb{R}$. This is done by following the same geometric steps as in \cite[subsection 5.3]{GHM3}, only applied to the bundle $E$ instead of the bundle $S$ used there. Where \cite[Theorem 4.13]{GHM3} was used, we now apply Proposition \ref{thm 1234}. In \cite{GHM3}, a homotopy invariance property of equivariant Callias-type operators \cite[Proposition 4.9]{GHM3} was proved, together with the fact that the index of a Callias-type operator does not change if one modifies the potential $\Phi$ on a cocompact subset \cite[Corollary 4.10]{GHM3}. The proofs of both of these properties carry over to the projective setting after making the modifications (1) -- (4) from above.

More precisely, the above discussion reduces the computation of $\Ind_{\Gamma,\sigma}(D+\Phi)$ to the $(\Gamma,\sigma)$-index of the operator \eqref{eq Callias on cylinder} below. To define this operator, let $E_{0,\sL}$ be as in \eqref{eq E0L}, and denote by $E^N_{0,\sL}$ its restriction to $N$. Let
$$E^{N\times \R}_{0,\sL,\pm} \to N \times \R$$ 
be the pullbacks of $E^N_{0,\sL,\pm}\to N$ along the canonical projection $N\times\R\to N$. Then $E^{N\times\R}_{0,\sL,\pm}$ are Clifford bundles over $T(N\times\R)$, with Clifford action
\[
\widehat c(v , t) = c(v + t\widehat n),
\]
where $v \in TN$, $t \in \R$, $\widehat n$ is the normal vector field to $N$ pointing into $M_+$, and $c$ is Clifford multiplication on $E_{0,\sL}$. Let $\nabla^{E^{N}_{0,\sL,\pm}}$ and $D^{E^{N}_{0,\sL,\pm}}$ be as in \eqref{eq connections} and \eqref{eq Dirac operators}. By pulling back $\nabla^{E^{N}_{0,\sL,\pm}}$ along $N\times\R\to\R$ and composing with $\widehat{c}$, we obtain Dirac operators $D^{E^{N\times\R}_{0,\sL,\pm}}_0$ acting on sections of $E^{N\times\R}_{0,\sL,\pm}$. In particular, operator $D^{E^{N\times\R}_{0,\sL,+}}_0$ is equivariant with respect to the pull-back of the projective action $T^{N,+}$ to $L^2(E^{N\times\R}_{0,\sL,+})$. Let $\chi \in C^{\infty}(\R)$ be an odd function such that $\chi(t) = t$ for all $t \geq 2$. Let $\chi_{N\times\R}$ be its pullback along the projection $N \times \R\to\R$. For our Callias-type operator, we will take two copies of the bundle $E^{N\times\R}_{0,\sL,+}$. Define
\[
D^{E^{N\times\R}_{0,\sL,+}} = 
\begin{pmatrix}
0 & D^{E^{N\times\R}_{0,\sL,+}}_0\\ D^{E^{N\times\R}_{0,\sL,+}}_0 & 0
\end{pmatrix},
\]
acting on smooth sections of $E^{N\times\R}_{0,\sL,+}\oplus E^{N\times\R}_{0,\sL,+}$. Then the endomorphism
\[
\chi^{N \times \R} = \begin{pmatrix}
0 & i\chi_{N\times\R} \\ -i \chi_{N\times\R} & 0
\end{pmatrix}
\]
is admissible for $D^{E^{N\times\R}_{0,\sL,+}}$ in the sense of Definition \ref{def Callias}, and
\begin{equation}
\label{eq Callias on cylinder}
D^{E^{N\times\R}_{0,\sL,+}}+\chi^{N \times \R}
\end{equation}
is a $(\Gamma,\sigma)$-invariant Callias-type operator on $N\times\R$. By the discussion in the first paragraph of this proof, we have
\begin{equation}
\label{eq main red}
\Ind_{\Gamma,\sigma}(D+\Phi)=\Ind_{\Gamma,\sigma}(D^{E^{N\times\R}_{0,\sL,+}}+\chi^{N \times \R}).
\end{equation}

It then suffices to prove that
\begin{equation}
\label{eq final red}
	\Ind_{\Gamma,\sigma}(D^{E^{N\times\R}_{0,\sL,+}}+\chi^{N \times \R})=\Ind_{\Gamma,\sigma}(D^{E^N_{0,\sL,+}}),
\end{equation}
which is the projective analogue of \cite[Proposition 5.7]{GHM3}. For this, note that the operator $D^{E^{N\times\R}_{0,\sL,+}}+\chi^{N \times \R}$ can be written explicitly as
$$
\begin{pmatrix}
0 & D^{E^N_{0,\sL,+}} \\ D^{E^N_{0,\sL,+}} & 0
\end{pmatrix}
\otimes 1_{C^{\infty}(\R)} + 
 \gamma_{E^N_{0,\sL,+}} \otimes
\begin{pmatrix}
0 & i\frac{d}{dt} \\  i\frac{d}{dt}  & 0
\end{pmatrix}
+
1_{C^{\infty}(E^N_{0,\sL,+})}\otimes
\begin{pmatrix}
0 &  -i\chi \\ i\chi & 0
\end{pmatrix},
$$
where $\gamma_{E^N_{0,\sL,+}}$ is a grading on $E^N_{0,\sL,+}$ defined as $-i$ times Clifford multiplication by the unit normal vector field on $N$ pointing into $M_+$. The equality \eqref{eq final red} then follows from the fact that the kernel of $ i\frac{d}{dt} \pm i\chi$ in $C^{\infty}(\R)$ is one-dimensional. 
Combining \eqref{eq main red} and \eqref{eq final red} concludes the proof.
\end{proof}
\begin{remark}
\label{rem any s}
Theorem \ref{thm Callias} continues to hold if we replace $\sigma$ by $\sigma^s$ for any $s\in\R$. In this case, the connection $\nabla^{\sL}$ from \eqref{eq nabla L} would be replaced by $\nabla^{\sL,s}=d+is\eta$.
\end{remark}

\subsection{Proof of Theorem \ref{thm main2}}
\begin{proof}[Proof of Theorem \ref{thm main2}]
This proof is similar to, but more subtle than, that of \cite[Theorem 2.1]{GHM3}, as it involves an additional scaling argument along with the use of an appropriate partition of unity. Hence we will give the full details.

First note that by a suspension argument, we only need to consider odd-dimensional $M$. In this case, let $\cS$ be the spinor bundle over $M$, let $\slashed{\partial}$ be the spin-Dirac operator, and let $\cS_\sL=\cS\otimes\sL$ for a $\Gamma$-trivial line bundle $\sL$. For $s\in\R$, let $\nabla^{\sL,s}$ be the Hermitian connection on $\sL$ defined by the one-form $is\eta$. In the notation of subsection \ref{subsection Callias localisation}, take $E_{0,\sL}=\cS_\sL$ and $D_0^s$ be the Dirac operator associated to the connection $\nabla^\cS\otimes 1+1\otimes\nabla^{\sL,s}$. Let
$$D^s=\begin{pmatrix}0&D_0^s\\D_0^s&0\end{pmatrix}$$
act on sections of $\cS_\sL\oplus\cS_\sL$. We now construct a potential $\Phi$ that is admissible for $D^s$, for all $s$, in the sense of Definition \ref{def Callias}.

Let $H$ be as in the statement of the theorem. Then $M\setminus H=X\cup Y$ for disjoint open subsets $X$ and $Y$. Pick a cocompact subset $K$ of $M$ such that $H\subseteq K\subseteq\overline X$ and $\kappa > 0$ on $K$, and the distance from  $X \setminus K$ to $Y$ is positive. 
Pick a $\Gamma$-invariant function 
$\chi \in C^{\infty}(M)$ such that $\chi$ equals $1$ on $Y$ and $-1$ on $X \setminus K$. Let $\Phi_0$ be the endomorphism of $\cS_\sL$ given by pointwise multiplication by $\chi$. Now define $D^s$ and $\Phi$ according to \eqref{eq D D0} and \eqref{eq Phi Phi0} respectively. Define the endomorphism
$$\Phi=\begin{pmatrix}0&i\chi\\-i\chi&0\end{pmatrix}.$$
One finds that $D_0^s=\spartial+isc(\eta)$. Together with the fact that $[isc(\eta),i\chi]=0$, this implies that
$$\{D^s,\Phi\}=-i\begin{pmatrix}sc(d\chi)&0\\0&-sc(d\chi)\end{pmatrix}.$$
By construction, the estimate $\Phi^2\geq\norm{\{D^s,\Phi\}}+1$ holds pointwise on $M\setminus K$, hence $\Phi$ is admissible for $D^s$, for all $s\in\R$, in the sense of Definition \ref{def Callias}.

Next, in the notation of subsection \ref{subsection proj Callias}, take $M_-=K$. Then $N=\partial M_-$ is a disjoint union $N_-\cup H$ for a cocompact subset $N_-$ such that $f|_{N_-} = -1$. In this case, 
\begin{equation}
\label{eq eigenbundle spin}
E^{N}_+ = \cS_\sL|_H.
\end{equation}
By Theorem \ref{thm Callias} and the proof of Theorem \ref{thm main1} (see also Remark \ref{rem any s}), it suffices to show that 
$\Ind_{\Gamma,\sigma}(D^s + \Phi) = 0$
for all sufficiently small $s$. For convenience, let us write $B_{\lambda}^s=D^s+\lambda\Phi$ for $\lambda>0$. By a homotopy argument, $\Ind_{\Gamma,\sigma}(B^s)=\Ind_{\Gamma,\sigma}(B^s_\lambda)$ for any positive $\lambda$, so it suffices to show that
$$\Ind_{\Gamma,\sigma}(B^s_\lambda) = 0$$
for some $\lambda>0$ and all sufficiently small $s$. Note that the endomorphism $\lambda\Phi$ is still admissible for $D^s$, and \eqref{eq eigenbundle spin} continues to hold. Let $K'$ be an arbitrary cocompact neighbourhood of $K$. By construction,
\begin{equation}
\label{eq M minus K}
(B^s_\lambda)^2\geq\lambda^2	
\end{equation}
on $M\setminus K$. On the set $K'$, we can obtain an estimate as follows. Letting $\nabla^s$ be the connection used to define $D^s$, we have
\begin{align}
\label{eq expansion}
(B^s_\lambda)^2&=(D^s)^2+\{D^s,\lambda\Phi\}+\Phi^2\nonumber\\
&={\nabla^s}^*\nabla^s+\frac{\kappa}{4}+is\lambda c(\omega)+\{D^s,\lambda\Phi\}+\lambda^2\Phi^2\nonumber\\
&\geq\frac{\kappa}{4}+is\lambda c(\omega)+\{D^s,\lambda\Phi\}+\lambda^2\Phi^2.
\end{align}
Let $\kappa_0=\inf_{x\in K}\kappa(x)>0$ by cocompactness of $K$. Then
\begin{itemize}
\item on $K$: there exist $s_0,\lambda_0>0$ such for all $s<s_0$ and $\lambda<\lambda_0$, the endomorphism \eqref{eq expansion} is bounded below by $\frac{\kappa_0}{8}$;
\item on $K'\setminus K$: since $\kappa\geq 0$ and $\{D^s,\lambda\Phi\}=0$, the endomorphism \eqref{eq expansion} is bounded below by $is\lambda c(\omega)+\lambda^2\Phi^2.$ By cocompactness of $K'$, there exist $s_1,\lambda_1$ such that $is\lambda c(\omega)+\lambda^2\Phi^2\geq\frac{\kappa_0}{8}$ for all $s<s_1$ and $\lambda<\lambda_1$.
\end{itemize}
Combining this with \eqref{eq M minus K}, we see that the estimate
$$(B^s_\lambda)^2\geq\frac{\kappa_0}{8}$$
holds on both $K'$ and $M\setminus K$ for all $s<\inf\{s_0,s_1\}$ and $\lambda<\inf\{\lambda_0,\lambda_1,\frac{\kappa_0}{8}\}$.
To combine these these estimates, let $\phi_1,\phi_2$ smooth functions $M\to [0,1]$ such that
\begin{itemize}
\item $\{\phi_1^2,\phi_2^2\}$ is a partition of unity on $M$;
\item $\supp(\phi_1)\subseteq K'$ and $\supp(\phi_2)\subseteq M\setminus K$.
\end{itemize}
For any $\varepsilon>0$, we may take $K'$ to be sufficiently large so that $\norm{d\phi_1}_\infty,\norm{d\phi_2}_\infty<\varepsilon$. Since $B^s$ is a perturbation of $\spartial$ by an endomorphism,
\begin{equation}
\label{eq commutators}
[\phi_i,B^s]=[\phi_i,\spartial]=c(d\phi_i),
\end{equation}
for $i=1,2$, hence $\phi_i B_{\lambda}^s=B_{\lambda}^s\phi_i+[\phi_i,\spartial]$. For any $u\in L^2(\cS_\sL)$, one computes that
\begin{align*}
\langle\phi_i B_{\lambda}^s u,\phi_i B_{\lambda}^s u\rangle&=\langle[\phi_i,\spartial]u,\phi_i B_{\lambda}^s u\rangle+\langle B_{\lambda}^s\phi_i u,\phi_i B_{\lambda}^su\rangle\\
&=\langle[\phi_i,\spartial]u,\phi_i B_{\lambda}^s u\rangle+\langle\phi_i B_{\lambda}^s u,[\phi_i,\spartial]u\rangle\\
&\qquad\qquad+\langle[\spartial,\phi_i]u,[\phi_i,\spartial]u\rangle+\langle B_{\lambda}^s\phi_i u,B_{\lambda}^s\phi_i u\rangle\\
&\geq\norm{B_{\lambda}^s\phi_i u}^2-2\varepsilon\norm{B_{\lambda}^s u}\norm{u}-\varepsilon^2\norm{u}^2,
\end{align*}
where the norms and inner products are taken in $L^2(\cS_\sL)$ and we have used that $\norm{\phi_i}_\infty\leq 1$. It follows that
\begin{align*}
\langle (B_{\lambda}^s)^2 u, u\rangle&=\langle\phi_1^2 B_{\lambda}^s u,B_{\lambda}^s u\rangle+\langle\phi_2^2 B_{\lambda}^s u,B_{\lambda}^s u\rangle\\
&\geq 2(\norm{B_{\lambda}^s\phi_i u}^2-2\varepsilon\norm{B_{\lambda}^s u}\norm{u}-\varepsilon^2\norm{u}^2)\\
&\geq\big(\frac{\kappa_0}{4}-\varepsilon^2\big)\norm{u}^2-4\varepsilon\norm{B_{\lambda}^s u}\norm{u}.
\end{align*}
Hence $(\norm{B_{\lambda}^s u}+2\varepsilon\norm{u})^2\geq\big(\frac{\kappa_0}{4}+3\varepsilon^2\big)\norm{u}^2,$ so that
$$\norm{B_{\lambda}^s u}\geq\Big(\sqrt{\frac{\kappa_0}{4}+3\varepsilon^2}-2\varepsilon\Big)\norm{u}.$$
Taking $\varepsilon$ small enough, and hence $K'$ large enough, we see that $(B_{\lambda}^s)^2$ is strictly positive. Thus $B^s_\lambda$ is invertible for $\lambda<\inf\{\lambda_0,\lambda_1,\frac{\kappa_0}{8}\}$ and $s<\inf\{s_0,s_1\}$, whence $\Ind_{\Gamma,\sigma}(B^s_\lambda) = 0$.
%
%
\end{proof}
\hfill\vskip 0.3in

\section{A quantitative obstruction in the non-cocompact setting}
\label{sec quantitative obstr}

When $M/\Gamma$ is non-compact, we can use quantitative $K$-theory to give obstructions to the existence of $\Gamma$-invariant metrics of positive scalar curvature on $M$. This uses the fact that the twisted Roe algebra is naturally filtered by propagation, making it an example of a \emph{geometric $C^*$-algebra}. We now review these concepts.
\subsection{Geometric $C^*$-algebras and quantitative K-theory}
\label{subsec qkt}
\begin{definition}
\label{def geometric algebra}
A unital $C^*$-algebra $A$ is \emph{geometric} if it admits a filtration $\{A_r\}_{r>0}$ satisfying the following properties:
\begin{enumerate}[(i)]
\item $A_r\subseteq A_{r'}$ if $r\leq r'$;
\item $A_r A_{r'}\subseteq A_{r+r'}$;
\item $\bigcup_{r=0}^\infty A_r$ is dense in $A$.
\end{enumerate}
\end{definition}
If $A$ is non-unital, then its unitization $A^+$, viewed as $A\oplus\mathbb{C}$ as as a vector space, is a geometric $C^*$-algebra with filtration $\{A_r\oplus\mathbb{C}\}_{r>0}.$ In addition, for each $n$, the matrix algebra $M_n(A)$ is a geometric $C^*$-algebra with filtration $\{M_n(A_r)\}_{r>0}.$



\begin{definition}[{\cite[Definition 2.15]{Chung}}]
\label{def quasi}
Let $A$ be a geometric $C^*$-algebra. For $0<\varepsilon<\frac{1}{20}$, $r>0$, and $N\geq 1$,
\begin{itemize}
\item an element $e\in A$ is called an \emph{$(\varepsilon,r,N)$-quasiidempotent} if
$$\norm{e^2-e}<\varepsilon,\qquad e\in A_r,\qquad\max(\norm{e},\norm{1_{A^+}-e})\leq N;$$
\item if $A$ is unital, an element $u\in A$ is called an \emph{$(\varepsilon,r,N)$-quasiinvertible} if $u\in A_r$, $\norm{u}\leq N$, and there exists $v\in A_r$ with
$$\norm{v}\leq N,\qquad\max(\norm{uv-1},\norm{vu-1})<\varepsilon.$$
The pair $(u,v)$ is called an $(\varepsilon,r,N)$\emph{-quasiinverse pair}.
\end{itemize}
\end{definition}
The quantitative $K$-groups $K_0^{\varepsilon,r,N}(A)$ and $K_1^{\varepsilon,r,N}(A)$ are defined by collecting together all quasiidempotents and quasiinvertibles over all matrix algebras, quotienting by an equivalence relation, and taking the Gr\"othendieck completion.
\begin{definition}[{\cite[subsection 3.1]{Chung}}]
\label{def quantitative K}
Let $A$ be a unital geometric $C^*$-algebra. Let $r>0$, $0<\varepsilon<\frac{1}{20}$, and $N>0$. 
\begin{enumerate}[label=(\roman*), leftmargin=0.29in]
\item Denote by $\textnormal{Idem}^{\varepsilon,r,N}(A)$ the set of $(\varepsilon,r,N)$-quasiidempotents in $A$. For each positive  integer $n$, let
$$\textnormal{Idem}_n^{\varepsilon,r,N}(A)=\textnormal{Idem}^{\varepsilon,r,N}(M_n(A)).$$
We have inclusions $\textnormal{Idem}_n^{\varepsilon,r,N}(A)\hookrightarrow\textnormal{Idem}_{n+1}^{\varepsilon,r,N}(A)$ given by
$e\mapsto\begin{pmatrix}e&0\\0&0\end{pmatrix}.$ Set 
$$\textnormal{Idem}_\infty^{\varepsilon,r,N}(A)=\bigcup_{n=1}^\infty\textnormal{Idem}_n^{\varepsilon,r,N}(A).$$
Define an equivalence relation $\sim$ on $\textnormal{Idem}_\infty^{\varepsilon,r,N}(A)$ by $e\sim f$ if $e$ and $f$ are $(4\varepsilon,r,4N)$-homotopic in $M_\infty(A)$. Denote the equivalence class of an element $e\in\textnormal{Idem}_\infty^{\varepsilon,r,N}(A)$ by $[e]$. Define addition on $\textnormal{Idem}_\infty^{\varepsilon,r,N}(A)/\sim$ by
$$[e]+[f]=\begin{bmatrix}e&0\\0&f\end{bmatrix}.$$
With this operation, $\textnormal{Idem}_\infty^{\varepsilon,r,N}(A)/\sim$ is an abelian monoid with identity $[0]$. Let $K_0^{\varepsilon,r,N}(A)$ denote its Grothendieck completion.

\item Denote by $GL^{\varepsilon,r,N}(A)$ the set of $(\varepsilon,r,N)$-quasiinvertibles in $A$. For each positive integer $n$, let
$$GL_n^{\varepsilon,r,N}(A)=GL_n^{\varepsilon,r,N}(M_n(A)).$$
We have inclusions $GL_n^{\varepsilon,r,N}(A)\hookrightarrow GL_{n+1}^{\varepsilon,r,N}(A)$ given by
$u\mapsto\begin{pmatrix}u&0\\0&1\end{pmatrix}.$ Set 
$$GL_\infty^{\varepsilon,r,N}(A)=\bigcup_{n=1}^\infty GL_n^{\varepsilon,r,N}(A).$$
Define an equivalence relation $\sim$ on $GL_\infty^{\varepsilon,r,N}(A)$ by $e\sim f$ if $u$ and $v$ are $(4\varepsilon,2r,4N)$-homotopic in $M_\infty(A)$. Denote the equivalence class of an element $u\in GL_\infty^{\varepsilon,r,N}(A)$ by $[u]$. Define addition on $GL_\infty^{\varepsilon,r,N}(A)/\sim$ by
$$[u]+[v]=\begin{bmatrix}u&0\\0&v\end{bmatrix}.$$
With this operation, $GL_\infty^{\varepsilon,r,N}(A)/\sim$ is an abelian group $[1]$.
\end{enumerate}
\end{definition}
\begin{remark}
\label{rem nonunital}	
If $A$ is a non-unital geometric $C^*$-algebra, then we have a canonical $*$-homomorphism $\pi\colon A^+\to\mathbb{C}$. Using contractivity of $\pi$, we have homomorphisms
$$\pi_*\colon K_i^{\varepsilon,r,N}(A^+)\to K_i^{\varepsilon,r,N}(\mathbb{C}),$$
where $i=0$ or $1$. Define $K_i^{\varepsilon,r,N}(A)=\ker(\pi_*)$.
\end{remark}
The following result on quasiidempotents and quasiinvertibles is useful.
\begin{lemma}[{\cite[Lemma 3.4]{GXY3}}]
\label{lem htpy}
Let $A$ be a geometric $C^*$-algebra. If $e$ is an $(\varepsilon,r,N)$-idempotent in $A$, and $f\in A_r$ satisfies 
$$\norm{f}\leq N,\qquad\norm{e-f}<\frac{\varepsilon-\norm{e^2-e}}{2N+1},$$ 
then $f$ is a quasiidempotent that is $(\varepsilon,r,N)$-homotopic to $e$. In particular, if
$$\norm{f}<\frac{\varepsilon}{2N+1},$$
then the class of $f$ is zero in $K_0^{\varepsilon,r,N}(A)$.

Suppose that $A$ is unital and $(u,v)$ is an $(\varepsilon,r,N)$-quasiinverse pair in $A$. If $a\in A_r$ satisfies 
$$\norm{a}\leq N,\qquad\norm{u-a}<\frac{\varepsilon-\max(\norm{uv-1},\norm{vu-1})}{N},$$ 
then $a$ is a quasiinvertible that is $(\varepsilon,r,N)$-homotopic to $u$. In particular, if
$$\norm{1-a}<\frac{\varepsilon}{N},$$
then the class of $a$ is zero in $K_1^{\varepsilon,r,N}(A)$.
\end{lemma}

There is a homomorphism of abelian groups
\begin{equation}
\label{eq kappa}	
\Psi\colon K_*^{\varepsilon,r,N}(A)\to K_*(A),
\end{equation}
preserving the $\Z_2$-grading, 
where $K_*^{\varepsilon,r,N}(A)$ (resp. $K_*(A)$) denotes the direct sum of the quantitative (resp. operator) $K_0$ and $K_1$-groups.

\vspace{0.1in}
\subsection{The quantitative higher index}
\label{subsec quantitative higher ind}
\hfill\vskip 0.05in
	\noindent Fix $0<\varepsilon<\tfrac{1}{20}$ and $N\geq 7$. For each $s\in\R$, let the algebras $\mathbb{C}[M;L^2(\cS_{\sL})]^{\Gamma,\sigma^s}$ and $C^*(M;L^2(\cS_\sL))^{\Gamma,\sigma^s}$ be as in Definition \ref{def twisted Roe}, with $\sigma$ and the projective representation $T$ replaced by $\sigma^s$ and $T^s$, as in Definition \ref{def proj action}. 
	
	For each $r>0$, define the subspace 
\begin{equation}
\label{eq filtration}
\mathbb{C}[M;L^2(\cS_\sL)]^{\Gamma,\sigma^s}_r\coloneqq\{T\in\mathbb{C}[M;L^2(\cS_\sL)]^{\Gamma,\sigma^s}\colon\textnormal{prop}(T)\leq r\}
\end{equation}
of $C^*(M;L^2(\cS_\sL))^{\Gamma,\sigma^s}$. Then with respect to the filtration 
$$\{\mathbb{C}[M;L^2(\cS_\sL)]^{\Gamma,\sigma^s}\}_{r>0},$$ 
$C^*(M;L^2(\cS_\sL))^{\Gamma,\sigma^s}$ is a geometric $C^*$-algebra in the sense of Definition \ref{def geometric algebra}. As in \cite{GXY3}, this structure allows us define a refinement of the higher index that takes values in the quantitative $K$-groups of $C^*(M;L^2(\cS_\sL))^{\Gamma,\sigma^s}$. The construction is similar to that in \cite[subsection 3.2.2]{GXY3}, so we will be brief.

The $(\Gamma,\sigma^s)$-invariant higher index $\Ind_{\Gamma,\sigma^s}(D)$ given in Definition \ref{def higher index} can be represented explicitly as follows. Let $\chi$ be a normalising function. If $\dim M$ is even, define the idempotent
\begin{equation}
\label{eq pchi}
p_{\chi}(D)=\begin{pmatrix}
			\left[(1-\chi(D)^2)^2\right]_{1,1} & \,\,\left[\chi(D)(1-\chi(D)^2)\right]_{1,2}\\[1ex]
			\left[\chi(D)(2-\chi(D)^2)(1-\chi(D)^2)\right]_{2,1} & \,\,\left[\chi(D)^2(2-\chi(D)^2)\right]_{2,2}
		\end{pmatrix},
\end{equation}
where the notation $[X]_{i,j}$ means the $(i,j)$-th entry of the matrix $X$. Then $\Ind_{\Gamma,\sigma}(D)$ is represented by the difference of idempotents
	\begin{equation}
	\label{eq Aj graded}
	A_\chi(D)=p_{\chi}(D)-\begin{pmatrix}0&0\\0&1\end{pmatrix}.
	\end{equation}
For $\dim M$ odd, $\Ind_{\Gamma,\sigma}(D)$ can be represented by the unitary
	\begin{equation}
	\label{eq Aj ungraded}
	A_\chi(D)=e^{\pi i(\chi+1)}(D).
	\end{equation}

\subsubsection*{Even-dimensional $M$}
\hfill\vskip 0.05in
	\noindent
	Choose a normalizing function $\chi$ such that 
	\begin{equation}
	\label{eq fourier r5}
	\supp\widehat{\chi}\subseteq\left[-\frac{r}{5},\frac{r}{5}\right].
	\end{equation}
	Let $A_{\chi}(D^s)=p_{\chi}(D^s)-\begin{psmallmatrix}0&0\\0&1\end{psmallmatrix}$ as in \eqref{eq Aj graded}. Then the \emph{$(\varepsilon,r,N)$-quantitative maximal higher index} of $D$ is the class
$$\Ind_{\Gamma,\sigma^s,L^2}^{\varepsilon,r,N}(D^s)=\left[p_{\chi}(D^s)\right]-\begin{bmatrix}0&0\\0&1\end{bmatrix}\in K_0^{\varepsilon,r,N}(C^*(M;L^2(\cS_\sL))^{\Gamma,\sigma^s}).$$

\subsubsection*{Odd-dimensional $M$}
\hfill\vskip 0.05in
	\noindent For each integer $n\geq 0$ define polynomials
	\begin{align}
	\label{eq gn}
	f_n(x)&=\sum_{k=0}^n\frac{(2\pi ix)^k}{k!},\qquad g_n(x)= f_n(x)-\left(\sum_{k=1}^n\frac{(2\pi i)^k}{k!}\right)x^2.	
	\end{align}
	One finds that as $n\to\infty$, the difference $e^{2\pi ix}-g_n(x)$
	converges uniformly to $0$ for $x$ in the interval $[-2,2]$. 
	Let $m=m(\varepsilon,N)$ be the smallest number such that
	\begin{align}
	\label{eq gm}
	 |g_m(x)g_m(-x)-1|<\varepsilon,\nonumber\\
	 |e^{2\pi i x}-g_m(x)|<1,
	\end{align}
	for all $x\in[-2,2]$. Pick a normalizing function $\chi$ satisfying
	\begin{equation}
	\label{eq fourier 1}
	\supp\widehat{\chi}\subseteq\left[-\frac{r}{\deg g_m},\frac{r}{\deg g_m}\right].
	\end{equation}
	Then the operator 
	$$S_\chi=\frac{\chi(D^s)+1}{2}$$ 
	has propagation at most $\frac{r}{\deg g_m}$ and spectrum contained in $[-\frac{1}{2},\frac{3}{2}]$.
	Then $g_m(S_\chi(D^s))$ is an $(\varepsilon,r,N)$-quasiinvertible in the unitisation of $C^*(M;L^2(\cS_\sL))^{\Gamma,\sigma^s}$. It was shown in \cite[subsection 3.2.2]{GXY3} that
	$$\Ind_{{\Gamma,\sigma^s}}(D^s)=[A_\chi(D^s)]=[g_m(S_\chi)]\in K_1(C^*(M;L^2(\cS_\sL))^{\Gamma,\sigma^s}).$$

The \emph{$(\varepsilon,r,N)$-quantitative higher index} of $D^s$ is the class
	$$\Ind_{{\Gamma,\sigma^s}}^{\varepsilon,r,N}(D^s)=[g_m(S_\chi)]\in K_1^{\varepsilon,r,N}(C^*(M;L^2(\cS_\sL))^{\Gamma,\sigma^s}).$$

\begin{remark}
First, although in the above constructions we needed to make a choice of $\chi$, the quantitative higher index obtained is independent of this choice. 	

Second, the $(\Gamma,\sigma)$-higher index of $D^s$ relates to its quantitative refinement by $\Ind_{\Gamma}(D^s)=\Psi(\Ind_{\Gamma}(D^s))$, where $\Psi$ is the homomorphism from \eqref{eq kappa}.
\end{remark}

\hfill\vskip 0.1in
\subsection{A quantitative obstruction}
\label{subsec quantitative obstr}
We now prove Theorem \ref{thm main3}. This uses the construction of the twisted higher index from subsection \ref{subsec twisted algebras} in terms of twisted Roe algebras, which are geometric $C^*$-algebras in the sense of \cite{Oyono2}. The result we obtain generalises \cite[Theorem 1.1]{GXY3}.
\begin{proof}
The technique of the proof is as in \cite[section 4]{GXY3}. The differences are that we now work with the reduced rather than the maximal version of the twisted Roe algebra, and that bounds on $\kappa$ used in that paper are now replaced by bounds on the endomorphism $\kappa+4isc(\omega)$. By Lemma \ref{lem Bochner}, we have
	$$(D^s)^2=\nabla^{s*}\nabla^s+\frac{\kappa}{4}+isc(\omega).$$
Suppose that
$\kappa+4isc(\omega)\geq C_s$ holds as an estimate on operators on $L^2(\cS_\sL)$. Let $\chi$ be a normalizing function whose distributional Fourier transform $\widehat{\chi}$ is supported on some finite interval $[-s,s]$ for $s>0$. 
For each $t>0$, let $\chi_t$ be the normalizing function defined by
\begin{equation}
\label{eq chit}
	\chi_t(u)=\chi(tu),
\end{equation}
	$u\in\mathbb{R}$. 
Let $A_{\chi}(D^s)$ be the index representative defined using $\chi$. 

If $M$ is even-dimensional, let
	\begin{equation*}
	A_{\chi}(u)\coloneqq
	\begin{pmatrix}
			(1-\chi(u)^2)^2 & \,\,\chi(t)(1-\chi(u)^2)\\[1ex]
			\chi(u)(2-\chi(u)^2)(1-\chi(u)^2) & \,\,\chi(u)^2(2-\chi(u)^2)-1
	\end{pmatrix},\quad u\in\mathbb{R}.
	\end{equation*}
Let $u_0>0$ and a function $\alpha$ be such that
	\begin{equation}
	\label{eq AChi small}
	\norm{A_{\chi}(u)}<\frac{\varepsilon}{2N+1}
	\end{equation}
	whenever $|1-\chi(u)^2|<\alpha(\varepsilon)$ for all $u$ such that $|u|>u_0$, where the norm of $A_{\chi}(u)$ is taken in $M_2(\mathbb{C})$.
	Note that for $N\geq 7$, \eqref{eq AChi small} also implies that $\norm{A_{\chi}^2(u)-A_{\chi}(u)}<\varepsilon$ if $|u|>u_0$. By \eqref{eq chit}, we have
	\begin{equation}
	\label{eq epsilonbeta}
	\Big|1-\chi_{\frac{2u_0}{\sqrt{C_s}}}(u)^2\Big|=\Big|1-\chi\left(\tfrac{2u_0 u}{\sqrt{C_s}}\right)^2\Big|<\alpha(\varepsilon)
	\end{equation}
	whenever $u\in\mathbb{R}\setminus(-\frac{\sqrt{C_s}}{2},\frac{\sqrt{C_s}}{2})$, while 
	$$\supp\Big(\widehat{\chi}_{\tfrac{2u_0}{\sqrt{C_s}}}(D^s)\Big)\subseteq\left[-\tfrac{2u_0}{\sqrt{C_s}}s,\tfrac{2u_0}{\sqrt{C_s}}s\right].$$
	It follows that $A_{\chi_{\frac{2u_0}{\sqrt{C_s}}}}(D^s)$ is an $(\varepsilon,\frac{10u_0}{\sqrt{C_s}}s,N)$-quasiidempotent in $2\times 2$-matrices over the unitisation of $C^*(M;L^2(\cS_\sL))^{\Gamma,\sigma^s}$ 
	with norm strictly less than $\frac{\varepsilon}{2N+1}$.
By Lemma \ref{lem htpy},
	\begin{align*}
	\Ind_{{\Gamma,\sigma^s,L^2}}^{\varepsilon,\frac{10u_0}{\sqrt{C_s}}s,N}(D^s)=
	0\in K_0^{\varepsilon,\frac{10u_0}{\sqrt{C_s}}s,N}(C^*(M;L^2(\cS_\sL))^{\Gamma,\sigma^s}).
	\end{align*}

Letting $\lambda_0=10u_0 s$, we obtain $\Ind_{\Gamma,\sigma^s,L^2}^{\varepsilon,\frac{\lambda_0}{\sqrt{C_s}},N}(D^s)=0$. Note that for any $r\geq\frac{\lambda_0}{\sqrt{C_s}}$, $\Ind_{\Gamma,\sigma^s,L^2}^{\varepsilon,r,N}(D^s)$ can also be represented by $A_{\chi_{\frac{2u_0}{\sqrt{C_s}}}}(D^s)$. The homomorphism
$$K_0^{\varepsilon,\frac{\lambda_0}{\sqrt{C_s}},N}(C^*(M;L^2(\cS_\sL))^{\Gamma,\sigma^s})\to K_0^{\varepsilon,r,N}(C^*(M;L^2(\cS_\sL))^{\Gamma,\sigma^s})$$
induced by the inclusion
$$\textnormal{Idem}_\infty^{\varepsilon,\frac{\lambda_0}{\sqrt{C_s}},N}(C^*(M;L^2(\cS_\sL))^{\Gamma,\sigma^s})\hookrightarrow \textnormal{Idem}_\infty^{\varepsilon,r,N}(C^*(M;L^2(\cS_\sL))^{\Gamma,\sigma^s})$$
takes $\Ind_{{\Gamma,\sigma^s,L^2}}^{\varepsilon,\frac{\lambda_0}{\sqrt{C_s}},N}(D^s)$ to $\Ind_{{\Gamma,\sigma^s,L^2}}^{\varepsilon,r,N}(D^s)$, which therefore vanishes.

If $M$ is odd-dimensional, let $m=m(\varepsilon,N)$ and the polynomial $g_m$ be as \eqref{eq gm}. Let $\chi$ be a normalizing function satisfying \eqref{eq fourier 1}, and let $s=\frac{r}{\deg g_m}$. Let $u_0>0$ be such that
	$$\norm{1-g_m(P_{\chi}(u))}<\frac{\varepsilon}{N}$$
	whenever $|1-\chi(u)^2|<\alpha(\varepsilon)$ holds for all $u$ such that $|u|>u_0$ or, equivalently, whenever
	\begin{equation}
	\label{eq epsilonbeta}
	\Big|1-\chi_{\frac{2u_0}{\sqrt{C_s}}}(u)^2\Big|=\Big|1-\chi\left(\tfrac{2u_0 u}{\sqrt{C_s}}\right)^2\Big|<\alpha(\varepsilon)
	\end{equation}
	for all $u\in\mathbb{R}\setminus(-\frac{\sqrt{C_s}}{2},\frac{\sqrt{C_s}}{2})$. Meanwhile,
	$$\supp\Big(\widehat{\chi}_{\tfrac{2u_0}{\sqrt{C_s}}}(D^s)\Big)\subseteq\left[-\tfrac{2u_0}{\sqrt{C_s}}s,\tfrac{2u_0}{\sqrt{C_s}}s\right].$$
	Thus $g_m(P_{\chi_{\frac{2u_0}{\sqrt{C_s}}}}(D^s))$ is an $(\varepsilon,\frac{2mu_0}{\sqrt{C_s}}s,N)$-quasiinvertible in $2\times 2$-matrices over the unitisation of $C^*(M;L^2(\cS_\sL))^{\Gamma,\sigma^s}$ satisfying	
\begin{equation}
\Big\|1-g_m(P_{\chi_{\frac{2u_0}{\sqrt{C_s}}}}(D^s))\Big\|<\frac{\varepsilon}{N}.
\end{equation}
By Lemma \ref{lem htpy},
	\begin{align*}
	\Ind_{{\Gamma,\sigma^s,L^2}}^{\varepsilon,\frac{2mu_0}{\sqrt{C_s}}s,N}(D^s) = 0\in K_1^{\varepsilon,\frac{2mu_0}{\sqrt{C_s}}s,N}(C^*(M;L^2(\cS_\sL))^{\Gamma,\sigma^s}).
	\end{align*}
Letting $\lambda_0=2mu_0 s$, we obtain $\Ind_{{\Gamma,\sigma^s},L^2}^{\varepsilon,\frac{\lambda_0}{\sqrt{C_s}},N}(D^s)=0$. For any $r\geq\frac{\lambda_0}{\sqrt{C_s}}$, the element $\Ind_{{\Gamma,\sigma^s,L^2}}^{\varepsilon,r,N}(D^s)$ can also be represented by $g_m(P_{\chi_{\frac{2u_0}{\sqrt{C_s}}}}(D^s))$. The homomorphism
$$K_1^{\varepsilon,\frac{\lambda_0}{\sqrt{c}},N}(C^*(M;L^2(\cS_\sL))^{\Gamma,\sigma^s})\to K_1^{\varepsilon,r,N}(C^*(M;L^2(\cS_\sL))^{\Gamma,\sigma^s})$$
induced by the inclusion
$$GL_\infty^{\varepsilon,\frac{\lambda_0}{\sqrt{C_s}},N}((C^*(M;L^2(\cS_\sL))^{\Gamma,\sigma^s})^+)\hookrightarrow GL_\infty^{\varepsilon,r,N}((C^*(M;L^2(\cS_\sL))^{\Gamma,\sigma^s})^+)$$
takes $\Ind_{{\Gamma,\sigma^s,L^2}}^{\varepsilon,\frac{\lambda_0}{\sqrt{C_s}},N}(D^s)$ to $\Ind_{{\Gamma,\sigma^s,L^2}}^{\varepsilon,r,N}(D^s)$, which therefore vanishes.
\end{proof}
\hfill\vskip 0.3in
\bibliographystyle{plain}
\bibliography{mybib}
\end{document}